\documentclass[a4paper,10pt,notitlepage,reqno]{article}
\usepackage[english]{babel}
\usepackage{amsmath,amssymb,amsthm,bm}
\usepackage{mathtools}
\usepackage{mathrsfs}
\usepackage{enumitem}
\usepackage{authblk}
\usepackage{cite}
\usepackage{xcolor}
\usepackage{tikz}
\usepackage{stackrel}
\usepackage{geometry}
\usepackage{hyperref}
\theoremstyle{plain} 
\newtheorem{theorem}{Theorem}[section]
\newtheorem{lemma}[theorem]{Lemma}
\newtheorem{proposition}[theorem]{Proposition}

\theoremstyle{definition}
\newtheorem*{definition}{Definition}

\theoremstyle{remark}
\newtheorem{remark}[theorem]{Remark}
\newtheorem{assumptionX}{Assumption}
\newenvironment{assumption}[1]
  {\renewcommand{\theassumptionX}{#1}\begin{assumptionX}}
  {\end{assumptionX}}

\numberwithin{equation}{section}

\DeclareMathOperator{\realpart}{Re}
\renewcommand{\Re}{\realpart}

\DeclareMathOperator{\imaginarypart}{Im}
\renewcommand{\Im}{\imaginarypart}

\DeclareMathOperator{\curl}{curl}

\newcommand{\vertiii}[1]{{\left\vert\kern-0.25ex\left\vert\kern-0.25ex\left\vert #1 
    \right\vert\kern-0.25ex\right\vert\kern-0.25ex\right\vert}}

\newcommand{\R}{\mathbb{R}}
\newcommand{\C}{\mathbb{C}}
\newcommand{\D}{\mathcal{D}}
\newcommand{\Z}{\mathbb{Z}}
\newcommand{\ip}[2]{\left\langle #1,\,#2 \right\rangle}

\newcommand{\normeq}[1]{{\left\vert\kern-0.25ex\left\vert\kern-0.25ex\left\vert #1 
    \right\vert\kern-0.25ex\right\vert\kern-0.25ex\right\vert}}

\renewcommand{\thefootnote}{\fnsymbol{footnote}}  

\usepackage{color}
\usepackage[normalem]{ulem}
\definecolor{DarkGreen}{rgb}{0,0.5,0.1} 

\newcommand\soutL{\bgroup\markoverwith
{\textcolor{DarkGreen}{\rule[.5ex]{2pt}{1pt}}}\ULon}

\newcommand{\Hm}[1]{\leavevmode{\marginpar{\tiny%
$\hbox to 0mm{\hspace*{-1.5mm}$\leftarrow$\hss}%
\vcenter{\vrule depth 0.1mm height 0.1mm width \the\marginparwidth}%
\hbox to
0mm{\hss$\rightarrow$\hspace*{-1.5mm}}$\\\relax\raggedright #1}}}

\newcommand\nnfootnote[1]{%
  \begin{NoHyper}
  \renewcommand\thefootnote{}\footnote{#1}%
  \addtocounter{footnote}{-1}%
  \end{NoHyper}
}

\title{\textbf{Relativistic Virial Operators }}

\author[1]{Lucrezia Cossetti} 
\author[1,2]{Luca Fanelli}		
\author[3,2]{Fabio Pizzichillo}

\affil[1]{Ikerbasque \& Universidad del Pa\'is Vasco/Euskal Herriko Unibertsitatea, UPV/EHU, \newline Aptdo. 644, 48080, Bilbao, Spain; lucrezia.cossetti@ehu.eus; luca.fanelli@ehu.es}

\affil[2]{BCAM-Basque Center for Applied Mathematics, Mazarredo, 14 E48009 Bilbao, Spain}

\affil[3]{Departamento de Matemática e Informática Aplicadas a la Ingeniería Civil y Naval, Universidad Politécnica de Madrid, E.T.S.I. de Caminos, Canales y Puertos, calle del Profesor Arangueren 3, 28040, Madrid, Spain; fabio.pizzichillo@upm.es}

\begin{document}

\date{\small 24 March 2026}
\maketitle




\begin{abstract}
	\noindent
	When studying Dirac operators, it is well known that the phenomenon of \emph{Zitterbewegung} leads to a lack of convexity of the variance, which creates difficulties in the analysis of dispersive properties. In particular, standard virial methods are harder to implement in the Dirac setting. In this paper, we introduce a new approach based on the center-of-energy operator, leading to a family of relativistic virial identities. As an application,
we establish spectral stability results for perturbed Dirac operators and prove local smoothing estimates for the associated evolution equation.

\end{abstract}

\nnfootnote{\emph{2020 Mathematics Subject Classification:81Q10, 81Q15, 35Q41}.}
\nnfootnote{\emph{Keywords}. Dirac operators, Relativistic virial identities, Absence of point spectrum, Local-smoothing estimates}


\section{Introduction}

In quantum mechanics, the evolution of a physical system of Hamiltonian $H$ is governed by the time-dependent Schr\"odinger equation
\begin{equation}\label{eq:schro}
  i \partial_t \psi(x, t) = H \psi(x, t),
\end{equation}
which describes the evolution of the quantum state (or wavefunction) $\psi.$

The complex-valued wavefunction $\psi(x, t) \in \mathbb{C}$ encodes all the probabilistic information about the state of a particle at position $ x $ and time $ t:$ the density $|\psi(x, t)|^2$ gives the probability of detecting the particle at the point $ x $ at time $ t ,$ assuming the wavefunction is normalized.

\smallskip
In the simplest case of a single particle moving in a region where a potential $V(x)$ is present, the Hamiltonian $H,$ which describes the total energy of the system, takes the form
\begin{equation}\label{eq:ham}
H =H_0 + V(x):= \frac{p^2}{2m} + V(x)
\end{equation}
where $ p = -i \nabla $ represents the momentum operator, and $ m $ is the mass of the particle. 
In the Heisenberg picture of quantum mechanics, one is interested in the evolution of the relevant \emph{observables} associated to the system (position, momentum, energy, spin, angular momentum, etc.)~and their \emph{expectation values}, rather than in the evolution of the state $\psi$ itself. The time derivative of the expectation value of an observable $A$ which does not depend on time is given by 
\begin{equation}\label{eq:heisenberg}
\frac{d}{dt} \langle A \rangle_\psi = \langle i [H, A] \rangle_\psi,
\end{equation}
where we have defined $\langle A\rangle_\psi:= \langle \psi, A \psi \rangle$  (note that our convention is that the inner product $\langle\cdot, \cdot \rangle$ is linear in the first component) and, as customarily, $[H, A] = HA - AH$ denotes the commutator of $ H $ and $ A.$ Equation~\eqref{eq:heisenberg} shows the central role that commutators play in the time evolution of physical quantities in quantum theory: if $ [H, A] = 0, $ then the observable $A$ is conserved, if not, its time evolution is governed by the commutator.

A fundamental example is given by the \emph{position} observable, corresponding to the choice $A = x_j$ in~\eqref{eq:heisenberg}, for $j = 1,\dots,d$. We have
\begin{equation*}
	\dot{x_j}=i[H, x_j]=\frac{p_j}{m}, 
	\qquad
	\dot{p_j}=i[H, p_j]=-\partial_j V(x), 
\end{equation*}
to be read as an identity for the expectations (\emph{cfr.}~\eqref{eq:heisenberg}). This then yields to the Ehrenfest's Theorem
\begin{equation*}
	m\ddot{x}=-\nabla V(x), 
\end{equation*}
which is the quantum analogue of Newton's law from classical mechanics.

For simplicity, from now on, we will set $m=\tfrac{1}{2},$ therefore~\eqref{eq:ham} reduces to 
\begin{equation*}
	H=H_0 + V(x), \qquad H_0=-\Delta.
\end{equation*}
In the description of the Schrödinger's dynamics another central role is played by the \emph{variance} of the position: 
considering $A=|x|^2,$ we denote $I(t):=\langle |x|^2 \rangle_{\psi}.$ Introducing the \emph{generator of dilations} operator 
\begin{equation}\label{eq:generator-dilations}
G:=\frac{1}{2}\{p,x\}
=\frac{p\cdot x+x\cdot p}{2}
=\frac{i}{4}[H_0,|x|^2]
=\frac{i}{4}[H, |x|^2],
\end{equation}
 one obtains the \emph{virial identity}
\begin{equation}\label{eq:virial-Schro}
	\dot{I}(t)=4\langle G \rangle_{\psi},
	\qquad 
	\ddot{I}(t)=4\langle i[H,G] \rangle_{\psi}
	=-\langle [H,[H,|x|^2]] \rangle_{\psi}
	=8\langle H_0 \rangle_{\psi} - 4\langle x\cdot \nabla V(x) \rangle_{\psi}. 
\end{equation}  
In particular, if $x\cdot \nabla V\leq 0,$ then $I(t)$ is convex, reflecting the dispersive nature of the Schrödinger evolution.
For stationary states $H\psi=E\psi,$ that is for states for which the time dependence is given through the multiplication by a phase factor, namely  $e^{itE}\psi,$ the variance is constant in time, hence identity~\eqref{eq:virial-Schro} reduces to the well known \emph{virial theorem}
\begin{equation}\label{eq:virial0}
2 \langle H_0 \rangle_\psi = \langle x \cdot \nabla V(x) \rangle_\psi,
\end{equation}
which expresses a balance between kinetic and potential contributions to the total energy (\emph{cf.}~the recent survey~\cite{CossettiKrejcirik2024}). In the case of homogeneous potentials $V(\lambda x)=\lambda^k V(x),$ it further specialises to  
$2 \langle H_0 \rangle_\psi = k \langle V \rangle_\psi.$

\medskip
These virial identities have proved to be fundamental in PDE analysis.
They yield Pohozaev-type constraints in stationary nonlinear models~\cite{Pohozaev}, serve to establish
blow-up criteria, as shown for the focusing NLS by Glassey and Zakharov~\cite{Glassey1977,Zakharov1972} and provide
Morawetz-type space-time bounds~\cite{Morawetz} as well as  local smoothing estimates
\cite{CS1988,Sjolin1987,Vega1988,BarceloRuizVega97,PerthameVega1999}. Due to their fundamental consequences, which extend well beyond the aforementioned results, virial identities are nowadays considered as an indispensable tool for understanding linear and nonlinear dynamics.

\medskip
When relativistic effects are taken into account, the Schrödinger Hamiltonian no longer suffices to give a realistic description of the system, as a matter of fact the Dirac Hamiltonian turns out to be the appropriate model. In space dimension $d\geq 1,$ this operator is defined as follows:
\begin{equation}\label{eq:Dirac}
 	H_D = \alpha\cdot p + m\beta, 
 	\qquad p=-i\nabla,
 	\qquad x\in \R^d,
\end{equation}
where $m\geq 0$ is the mass parameter, and $\alpha=(\alpha_1,\dots,\alpha_d)$ and $\beta$ are
Hermitian matrices acting on $\mathbb{C}^N$, with $N=2^{\lfloor (d+1)/2\rfloor}$, satisfying the Clifford algebra relations
\begin{equation}\label{eq:commutation-rel-dirac-matr}
\{\alpha_j,\alpha_k\}=2\delta_{jk}I, \qquad \{\alpha_j,\beta\}=0, \qquad \beta^2=I.
\end{equation}
The integer $N$ coincides with the dimension 
of the irreducible complex representation of the Clifford algebra $\mathrm{Cl}_{1,d}(\R)$ (\emph{e.g.} $N=2$ for $d=1,2$; $N=4$ for $d=3,4$; $N=8$ for $d=5,6$ etc.).
For the existence of such matrices, see \cite{Friedrich2000}.

Accordingly, the evolution of the wavefunction $\psi\colon\R\times\R^d\to\C^N$ is governed by the Dirac equation
\begin{equation}\label{eq:free-dirac-evolution}
 i\partial_t \psi = H_D \psi.
\end{equation}

The wavefunction $\psi(x,t)$ is a $N$-component spinor,  this reflects the internal spin and the particle-antiparticle nature inherent to the relativistic theory.
%

The relativistic dynamics displays remarkable new features. In particular, one observes that the observable position does not behave classically: 
Indeed, choosing $A = x_j$, for $j = 1,\dots,d$, in~\eqref{eq:heisenberg}, and letting $\psi$ be a solution to~\eqref{eq:free-dirac-evolution}, we obtain
\begin{equation*}
\dot{x_j}= i[H_D, x_j] = \alpha_j,
\qquad 
\dot{\alpha_j}=i[H_D,\alpha]=2i(p_j-H_D\alpha_j).
\end{equation*}
This identity indicates that the velocity operator is not given in terms of momentum $p$, as in classical relativistic kinematic, but rather by the matrix-valued operator $\alpha$. Furthermore, since the velocity operator does not commute with the Hamiltonian $H_D$, the acceleration is non zero. One can further see that 
\begin{equation*}
\alpha(t)   = H_D^{-1} p 
      + \big(\alpha - H_D^{-1} p \big)\, e^{-2i H_D t}.
\end{equation*}
Thus, the velocity rapidly oscillates around the mean value $H_D^{-1} p \sim \frac{p}{E_p}$ (which is just the \emph{classical} velocity operator) with frequency $2E_p,$ where $E_p:=\sqrt{|p|^2 +m^2},$ leading to the phenomenon of \emph{Zitterbewegung} (trembling motion). This phenomenon is known to arise due to the interference between positive and negative energy components of the Dirac spinor (more details can be found in~\cite[Ch.1.6]{Thaller92-book}). 

Looking at the evolution of the position operator, one sees that
\begin{equation}\label{eq:position}
x(t) = x(0) + p H_D^{-1} t + \frac{i}{2} H_D^{-1} \left( \alpha(0) - pH_D^{-1} \right) e^{-2iH_D t}.
\end{equation}
As a consequence of this fact the variance $I(t):=\langle |x|^2\rangle_\psi$ does not satisfy a simple
convexity law as in~\eqref{eq:virial-Schro}.

Nevertheless, for stationary eigenstates of $H=H_D+V$, \emph{i.e.} for states $\psi$ such that $H\psi = E\psi,$ one can still deduce a
virial identity. Indeed, using that for any observable $A$ one has the operator relation $\langle\psi,[H,A]\psi\rangle=0$ and then choosing $A=G$ gives
\begin{equation}\label{eq:virial-rel}
 \langle \alpha\cdot p\rangle_{\psi} = \langle x\cdot\nabla V\rangle_{\psi},
\end{equation}
which is the natural relativistic analogue of~\eqref{eq:virial0} expressing the balance between the relativistic kinetic energy and the potential energy. As a matter of fact, in the non-relativistic limit (regime where $\alpha \cdot p \sim p^2/m$), identity~\eqref{eq:virial-rel} reduces to~\eqref{eq:virial0}. Nevertheless,  unlike the Schrödinger case, one is not able to recover identity~\eqref{eq:virial-rel} from the dynamics of the variance, but only from the algebraic structure of the Dirac operator. In other words, we cannot read~\eqref{eq:virial-rel} as a second time derivative or a double commutator identity for $H,$ as it is the case for~\eqref{eq:virial-Schro}. Moreover, note that since $ \langle \alpha\cdot p\rangle_{\psi}$ has no definite sign, even if $x\cdot \nabla V\leq 0,$ the quantity $\langle\psi,[H,A]\psi\rangle$ fails to be positive, in contrast to the Schrödinger case. 
The preceding considerations make it clear that virial methods are more difficult to implement in the Dirac setting and, in principle, must be adapted to the distinctive algebraic structure of the Dirac operator (we refer to the recent result~\cite{HMM2025,HMM2026}, where new $L^2$-based virial-type identities are developed and employed to study long-time behavior of solutions to nonlinear Dirac equations).

\medskip
The aim of this paper is to introduce new relativistic virial-type identities that admit a dynamical interpretation and that lead naturally to the investigation of new observables that commute positively with the Dirac operator.

To this purpose, let us start with a simple but crucial algebraic
identity: for any couple of operators $T,A$ (we disregard for now domain issues), we have
\begin{equation}\label{eq:mainmain}
 [T,\{T,A\}] = [T^2,A],
\end{equation}
where $\{T,A\}=TA+AT$ is the anti-commutator between $T$ and $A.$
Specializing~\eqref{eq:mainmain} to $T=H_D$, one finds
\begin{equation}\label{eq:algebraic}
 [H_D,\{H_D,A\}] = [H_D^2,A] =\big[H_0,A\big],
\end{equation}
where we have used the well-known super-symmetry property of the free Dirac operator, namely $H_D^2=(H_0 + m^2)I.$

The commutator-anti-commutator identity~\eqref{eq:algebraic} expresses that the \emph{relativistic} evolution of $\{H_D,A\}$ mirrors the \emph{non-relativistic}
evolution of $A$ (see~\eqref{eq:heisenberg}). This suggests that relativistic virial identities need to be understood as \emph{first derivative} of suitable observable driven by the non-relativistic dynamic.

The intuition behind identity~\eqref{eq:algebraic} comes from the so-called \emph{center-of-energy} operator $N=\{H_D,x\},$ which measures the position weighted by the energy distribution, in contrast with the usual position operator which refers only to probability distribution, and from the commutation relation of this operator with the free Dirac operator $H_D$ (see~\cite[Appendix 1.C]{Thaller92-book}), namely
\begin{equation*}
	[H_D,N]=[H_0,x]
	\quad \text{or, equivalently,} \quad
	[H_D, \{H_D,x\}]=[H_0,x].
\end{equation*}  
Identity above shows that the relativistic evolution of $N$ is driven by the non-relativistic velocity (\emph{cf.}~\eqref{eq:heisenberg}). 

Now, due to the fundamental \emph{positive}-commutator identity
\begin{equation}\label{eq:positive-commutator}
	[H_0, iG]=2H_0,
\end{equation}
(\emph{cf.}~\eqref{eq:virial-Schro}), it is natural to search for the observable whose relativistic evolution corresponds to the non-relativistic evolution of the generator of dilations $G$. In other words, one aims to find an operator $\mathsf{X}$ verifying
\begin{equation}\label{eq:question}
	[H_D, \mathsf{X}]=[H_0,iG].
\end{equation}
Relation~\eqref{eq:algebraic} already contains the answer, namely $\mathsf{X}=\{H_D,iG\},$ and leads to the introduction of a new relativistic virial-type operator. 

\begin{definition}[Relativistic virial operator]\label{def:L}
Let $H_D$ be the Dirac operator defined in~\eqref{eq:Dirac} and let $G$ be the generator of dilations operator defined in~\eqref{eq:generator-dilations}. We define the \emph{relativistic virial operator} as the solution to~\eqref{eq:question}, namely the operator 
\begin{equation}\label{eq:L}
L:=\{H_D,iG\}.
\end{equation}
The operator $L$ truly plays the role of a relativistic virial operator, intertwining, by construction, the Dirac dynamics with dilation symmetries. 
To the best of our knowledge, such an operator does not appear in the standard literature, and its definition seems to be new. 

The new relativistic virial operator $L,$ defined in~\eqref{eq:L}, can be rewritten alternatively as 
\begin{equation*}
	L=-\tfrac{1}{4}\left\{H_D,\left[H_0,|x|^2\right]\right\}.
\end{equation*}
Thus, one can even define the following generalized version of~\eqref{eq:L}, namely
\begin{equation}\label{eq:L2}
L_\phi:=-\tfrac{1}{4}\left\{H_D,\left[H_0,\phi\right]\right\}
=\{H_D,iG_\phi\},
\end{equation}
for a given function $\phi:\R^d\to\R$ and where we have defined
\begin{equation}\label{eq:iGphi}
	iG_\phi:=-\tfrac{1}{4}[H_0,\phi].
\end{equation}

Observe that both $L$ and $L_\phi$ defined above are second-order differential operators. 
\end{definition}

The following identities are immediate consequences of identities~\eqref{eq:algebraic} and~\eqref{eq:positive-commutator}.
\begin{proposition}
	Let $L$ and $L_\phi$ be defined as in~\eqref{eq:L} and~\eqref{eq:L2}, respectively. Consider $H=H_D+ V,$ where $H_D$ is defined as in~\eqref{eq:Dirac}, one has
	\begin{align}
	\label{eq:spectral}
		[H,L] &
		= 2H_0+[V,L],\\
		\label{eq:morawetzid}
		\left[H,L_\phi\right] & 
		=-\tfrac{1}{4}\left[H_0,[H_0,\phi]\right]+\left[V,L_\phi\right].
\end{align}
\end{proposition}
We emphasize once more that the two identities above are to be understood on the appropriate operator domains, where every term involved is well defined. 


Identities~\eqref{eq:spectral} and~\eqref{eq:morawetzid} form the core of this manuscript:~\eqref{eq:spectral} will provide valuable information on the point spectrum of $H$, while~\eqref{eq:morawetzid} will yield purely real-variable proofs of Morawetz-type and local smoothing estimates for the Dirac evolution equation, without the need to square the full Hamiltonian $H$, but only the free operator $H_D.$

\medskip
Before presenting our main results, we introduce the following assumptions, which will be required throughout the paper. More specifically, hypothesis~\ref{ass:HS} will be needed for the results in the \emph{stationary} setting, while hypothesis~\ref{ass:HE} will be used for the associated \emph{evolution} problem.
\begin{assumption}{(HS)}\label{ass:HS}
Let $H = H_D + V$, where $H_D$ represents the free Dirac operator defined in~\eqref{eq:Dirac} and $V:\R^d \to \C^{N\times N}$ is a Hermitian matrix-valued potential such that $V\in C^\infty(\R^d\setminus\{0\})$ and $V$ is bounded outside any ball containing the origin.
We assume that $H$ is essentially self-adjoint on $C^\infty_c(\R^d;\C^N)$ and self-adjoint on $L^2(\R^d;\C^{N})$ with domain 
$\mathcal D(H)\subseteq H^1(\R^d;\C^{N})$.
\end{assumption}

\begin{assumption}{(HE)}\label{ass:HE}
Let $V:\R^d \to \C^{N\times N}$ be a Hermitian matrix-valued potential. We assume that there exists $a\in[0,1)$ such that for any $\psi\in \D(H)$:
\begin{equation}\label{eq:smallness-potentia}
\|V\psi\|_{L^2}\leq a\|\nabla\psi\|_{L^2}
\quad
a\in[0,1).
\end{equation}
\end{assumption}

\begin{remark}\label{rem:self}
By the Kato--Rellich theorem, Assumption~\ref{ass:HE} is a sufficient but not necessary condition for Assumption~\ref{ass:HS}.
Indeen, Assumption~\ref{ass:HE} also is fulfilled in several classical situations:
If $V\in L^\infty(\R^d;\mathcal \C^{N\times N})$ is Hermitian, then $H=H_D+V$ is self-adjoint on the domain of the free Dirac operator and, in particular, on $H^1(\R^d;\C^{N}).$ This follows from the Kato--Rellich theorem since $V$ is $H_D$–bounded with relative bound $0$ (see, e.g., \cite[Ch.~4]{Thaller92-book}). More generally, the same conclusion holds whenever $V$ is $H_D$–relatively bounded with relative bound strictly less than $1$.

However, conditions based on $H_D$–relative smallness are not optimal in the Dirac setting. This becomes apparent already in the case of the three-dimensional \emph{electrostatic} Coulomb potential, namely $V(x)=-\nu\,|x|^{-1}I_4$.
Indeed, by the Hardy inequality, the potential $V$ is $H_D$ relatively bounded with relative bound $2|\nu|$, and therefore the Dirac–Coulomb operator is self-adjoint on $H^1(\R^3;\C^{4})$ whenever $|\nu|<1/2$. 
However, this range is not optimal: the Dirac–Coulomb operator is self-adjoint on $H^1(\R^3;\C^{4})$ if and only if $|\nu|< \sqrt{3}/2$ (see \cite[Ch.~4]{Thaller92-book}). 
For $\sqrt{3}/2<|\nu|\leq 1$, the operator admits infinitely many self-adjoint extensions, and a canonical (``distinguished'') extension can be selected, for instance, via finite potential energy or gap methods~\cite{DES2023,ELS2021}. 
For $|\nu|> 1$, further pathologies arise and a more delicate analysis is required (see the references above).

Other relevant examples are provided by the \emph{Lorentz-scalar} Coulomb potential $V(x)=-\mu\,|x|^{-1}\beta$ and by the \emph{anomalous magnetic} Coulomb potential $V(x)=\tau(-i\alpha\cdot x \beta)|x|^{-2}$.
In the first case, the operator $H$ is self-adjoint on $H^1(\R^3;\C^4)$ for all $\mu\in\mathbb{R}$, while in the second case self-adjointness holds for $|\tau|<1/2$ (see \cite{CassanoPizzichillo18}).
\end{remark}

\medskip
We are now ready to state the main results of this manuscript. 
The first two theorems concern the stationary eigenvalue problem associated with
\begin{equation}\label{eq:dirac-evolution}
i\partial_t \psi=(H_D+V)\psi
\end{equation}

More specifically, they provide smallness conditions in \emph{critical functional scales} on $V$ under which the operator $H=H_D + V$ does not admit eigenvalues. We treat separately the massless and massive case, as they exhibit substantially different features.

\begin{theorem}[Massless Dirac with potential: absence of point spectrum]\label{thm:spectral-stability}
Let $d\geq3$, $m=0$, and let Assumption~\ref{ass:HS} be satisfied.
Assume that there exist two positive constants $C_1$ and $C_2$ such that, for any $j=1,\dots,d,$ one has
\begin{equation}\label{eq:ifin}
|\nabla V(x)|\leq\frac{C_1}{|x|^2},
\qquad
\left|\left[\alpha_j,V\right]\right|\leq \frac{C_2}{|x|}.
\end{equation}
Then, if $C_1$ and $C_2$ satisfy 
\begin{equation*}
\frac{4C_1}{d-2}+C_2<1,
\end{equation*}
it follows that $\sigma_p\left(H_D+V\right)=\varnothing.$
\end{theorem}

In the massive case, we have the following generalization of the previous result.

\begin{theorem}[Massive Dirac with potential: absence of point spectrum]\label{thm:spectral-stability2}
Let $d\geq3$, $m>0$, and let Assumption~\ref{ass:HS} be satisfied.
Assume that there exist four positive constants  $C_1,C_2,C_3$ and $C_4$ such that, for any $j=1,\dots,d,$ one has
\begin{equation}\label{eq:ifin2}
|\nabla V(x)|\leq\frac{C_1}{|x|^2},
\qquad
\left|\left[\alpha_j,V\right]\right|\leq \frac{C_2}{|x|},
\qquad
\Big|   (x\cdot \nabla)\big[ \{V,\beta\}\big]\Big|\leq \frac{C_3}{|x|^2}
\qquad
|[\beta,V]|\leq\frac{C_4}{|x|^2}.
\end{equation}
Then, if $C_1, C_2, C_3$ and $C_4$ satisfy 
\begin{equation}\label{eq:tesi1}
\frac{4C_1}{d-2}+C_2+\frac{2mC_4}{d-2} + \frac{2mC_3}{(d-2)^2}<1,
\end{equation}
it follows that $\sigma_p\left(H_D+V\right)=\varnothing.$
\end{theorem}

\begin{remark}[Critical functional scales]\label{rem:ass2}
The smallness assumptions in Theorems~\ref{thm:spectral-stability} and~\ref{thm:spectral-stability2}
reflect the critical functional scales for the Dirac operator.
For \emph{homogeneous} potentials, the conditions in~\eqref{eq:ifin} in Theorem~\ref{thm:spectral-stability} (\emph{massless} case) single out the Coulomb scale
$V(x)\sim\frac{c}{|x|},$
which is critical for the massless Dirac operator.
In dimension $d=3$, for electrostatic Coulomb-type potentials, namely $V(x)=(\tfrac{a}{|x|} + b)I$ one has $[\alpha_j,V]=0$ for all $j$, so that~\eqref{eq:ifin} reduces to
the gradient bound $|\nabla V|\le \frac{a}{|x|^2}$ and~\eqref{eq:tesi1} is satisfied whenever $a<\tfrac14$ (for general dimensions $d\geq 3,$ condition~\eqref{eq:tesi1} provides a clear dimension-dependent smallness condition on the coupling constant). 
Thus, Theorem~\ref{thm:spectral-stability} recovers the fact that small Coulomb couplings
produce no point spectrum in the massless case. We stress that, when $[\alpha_j,V]=0$ (\emph{e.g.} when the potential is electrostatic),  no additional decay of $V$ at infinity is required. 

We emphasize that our result does not contradict the explicit zero-energy eigenfunctions constructed by Loss and Yau in~\cite{LossYau1986} (see also~\cite{BalinskyEvans2001}) for the Weyl-Dirac operator with certain potentials decaying like $\langle x\rangle^{-2}.$ Indeed, in their example the quantity $\sup |x|^2|\nabla V(x)|$ is of order one (not small as required here). This shows that when the constants in~\eqref{eq:ifin} are large, eigenvalues (in particular, zero modes) \emph{do} occur. By contrast, Theorem~\ref{thm:spectral-stability} asserts that if the smallness condition~\eqref{eq:tesi1} is satisfied, no point spectrum can appear.
The theorem therefore delineates a sharp qualitative dichotomy between the \emph{small Coulomb regime} (no eigenvalues) and the \emph{large Coulomb regime} (possible zero modes).

\medskip
For the \emph{massive} Dirac operator, where scaling invariance is lost, two distinct cases must be considered. 

In the more general case where $V$ does not anticommute with $\beta$, two distinct critical regimes  arise in the assumptions~\eqref{eq:ifin2}: the critical behavior at infinity is governed by the
inverse-square scale, while the critical behavior at the origin is given by the Coulomb potential, namely
$V(x)\sim\frac{c}{|x|+|x|^2}.$
In this case, the absence of eigenvalues cannot be expected for Coulomb-type potentials: one always expects (and in fact knows from Dirac’s computation~\cite{D}) that eigenvalues occur inside the spectral gap $(-m,m)$, corresponding to the discrete levels of the relativistic hydrogen atom.
The need to account simultaneously for both Coulomb and inverse-square scales has been
emphasized in several works, including Cuenin, Laptev and Tretter~\cite{CueninLaptevTretter14},
D'Ancona and Fanelli~\cite{DAnconaFanelli08}, and Fanelli and Krej\v{c}i\v{r}\'ik~\cite{FanelliKrejcirik19},
among others. These papers develop limiting absorption principles, dispersive estimates, and
spectral stability criteria that all highlight the special role of the Coulomb and
inverse-square thresholds.

Nevertheless, in the case where $V$ anticommutes with the Dirac matrix $\beta$, the Coulomb scale becomes admissible, despite the obstruction described above.
A physically relevant example is provided by \emph{anomalous magnetic potentials}, defined by $V=-i\beta\alpha\cdot\nabla \varphi$, where $\varphi:\mathbb{R}^d\to \mathbb{C}$ is a suitable scalar function.
In the $3$-dimensional case, for $\varphi=\tau \log|x|$ with $\tau\in\mathbb{R}$, one recovers the Coulomb anomalous magnetic potential $V=-i\tau\beta\alpha\cdot \frac{x}{|x|^2}$ that clearly verifies~\eqref{eq:ifin2}. Hence, Theorem~\ref{thm:spectral-stability2} ensures the absence of eigenvalues whenever the smallness condition~\eqref{eq:tesi1} is fulfilled.
This is consistent with the results of \cite[Remark 1.12]{CassanoPizzichillo18}, where eigenvalues and eigenfunctions are computed explicitly. 
Indeed, a careful inspection of this result reveals  that the corresponding eigenfunctions are not $L^2$-integrable whenever $|\tau|<1/2$, which is precisely the regime corresponding to Assumption~\ref{ass:HS}.
\end{remark}

\smallskip
\begin{remark}[Comparison with existing results]
Standard virial methods for Dirac operators read as follows (see~\cite[Thm.~4.1.2]{Albeverio72} and~\cite[Cor.~4.22]{Thaller92-book}): let $G$ be the generator of dilations defined in~\eqref{eq:generator-dilations}, for any solution $\psi$ of the eigenvalue equation $H\psi=\lambda \psi,$ with $H=H_D + V,$ one has
\begin{equation*}
	0=\ip{\psi}{[H,iG]\psi}
=\ip{\psi}{-i\alpha \cdot \nabla \psi}-\ip{\psi}{(x\cdot \nabla V)\psi}
=\lambda-m \ip{\psi}{\beta\psi}-\ip{\psi}{(V+ x\cdot\nabla V)\psi},
\end{equation*}
where, in the second identity, we have used~\eqref{eq:virial-rel}. Thus
\begin{equation}\label{eq:virial-identity-full}
	\lambda-m \ip{\psi}{\beta\psi}=\ip{\psi}{(V+ x\cdot\nabla V)\psi}.
\end{equation}
 The potential term on the right-hand side can be controlled, via the Hardy inequality, provided that $|V+x\cdot\nabla V|\lesssim |x|^{-1}$, and it is small when the corresponding coupling constant is small.
Therefore, standard virial arguments require \emph{both} $V$ and $\nabla V$ to be small at the Coulomb scale.
Moreover, they typically yield absence of eigenvalues only in restricted spectral regions. 
For instance, it easily follows from~\eqref{eq:virial-identity-full} that if $V(x)+x\cdot\nabla V$ has only non-positive eigenvalues, then $\lambda\leq m.$ Similarly, positivity of the matrix $V(x)+x\cdot\nabla V$ implies $\lambda\geq -m.$

Not relying exclusively on virial-type arguments, the question of whether perturbed Dirac
operators admit eigenvalues (in particular, embedded eigenvalues) has been extensively
investigated. In~\cite{Roze1970}, Roze proved that the Dirac operator $\alpha\cdot(p+A(x)) + \beta + q(x)I$
has no eigenvalues embedded in the essential spectrum for potentials decaying faster than
the Coulomb scale at infinity, namely under the assumption
\begin{equation}\label{eq:condA}
\lim_{|x|\to\infty}\bigl(|A(x)|+|q(x)|\bigr)=0.
\end{equation}
Later, Kalf~\cite{Kalf_1981} replaced the physically inadequate condition~\eqref{eq:condA} by
the gauge-invariant requirement
\begin{equation}\label{eq:condB}
\lim_{|x|\to\infty}\bigl(|B(x)|+|q(x)|\bigr)=0, 
\qquad B=\curl A.
\end{equation}
In the same paper, he also replaced condition~\eqref{eq:condB} (which excludes Coulomb-type
potentials) by an alternative assumption that allows Coulomb potentials, at the price of
imposing a sign condition on the radial derivative of $q$ (see conditions (iii) and (iv) in
\cite[Thm.~1]{Kalf_1981}).  

Vogelsang~\cite{Vogelsang1987} improved this result by requiring only some decay of the negative
part of the radial derivative of the potential (note that no magnetic field is involved in
this setting). A similar result to those of~\cite{Kalf_1981} and~\cite{Vogelsang1987} was
proved by Berthier and Georgescu in~\cite{BerthierGeorgescu1987}, extending beyond the
purely electrostatic case. Their potential $V=L+W$ is decomposed into a long-range part $L$
and a short-range part $W$: the long-range component $L$ satisfies local boundedness and a
sign condition on its radial derivative and, in particular, includes the Coulomb potential
(see \cite[(38), (39)]{BerthierGeorgescu1987}), while the short-range part $W$ can be singular
and decays at infinity as a short-range perturbation (see \cite[(40), (57)]{BerthierGeorgescu1987}).  

For further results in this direction, see
\cite{BoussaidComech2016,GeorgescuMantoiu2001,ErdoganGoldbergGreen19} and the references
therein.  Other recent results in a purely magnetic context can be found in the recent papers~\cite{HundertmarkKovarik2024,CossettiFanelliKrejcirik20}.

In~\cite{ArrizabalagaCossettiMorales2025}, the authors obtained results that are very close to our Theorems~\ref{thm:spectral-stability} and~\ref{thm:spectral-stability2} above. The strategy consists in squaring the full Hamiltonian $H$ and then handling the \emph{first-order} terms the arise from this procedure. As a consequence, their approach requires stronger assumptions on the perturbation than those in Theorems~\ref{thm:spectral-stability} and~\ref{thm:spectral-stability2}. This is due to the fact that in our proof we need the squaring of the free part of the operator only.  This difference between the results is already  apparent in the massless case for the prototypical example of electrostatic Coulomb potentials, namely $V(x)=-\nu |x|^{-1} I_N.$ Indeed, Theorem 1.4 in~\cite{ArrizabalagaCossettiMorales2025} establishes absence of point spectrum under the condition 
	$\frac{4\nu}{(d-2)} + \frac{4\nu^2}{(d-2)^2}<1,$
whereas our Theorem~\ref{thm:spectral-stability} only requires
	$\frac{4\nu}{(d-2)}<1.$
\end{remark}

\smallskip
\begin{remark} [Relation to non-selfadjoint results]
For non-self-adjoint perturbations of $H_D$, sharp \emph{localization} bounds for eigenvalues have been obtained under integrability assumptions 
(see, e.g.,~\cite{CueninLaptevTretter14, FanelliKrejcirik19}). 
Our theorems address the \emph{self-adjoint} case and provide absence of point spectrum under scale-critical, matrix-commutator conditions; 
in this sense they are complementary and highlight the role of genuine first-order Dirac structure via the operator $L=\{H_D,iG\}$. 
We strongly believe that our method should also apply to the non-self-adjoint setting, 
a question we plan to address in future work.
\end{remark}

\medskip
The new relativistic virial identities also yield significant information on the evolution
equation~\eqref{eq:dirac-evolution}. More precisely, this technique allows us to establish a
local smoothing effect under suitable assumptions on the potentials. We present two
distinct results: one in the physical dimension $d=3,$ and another for higher dimensions $d\geq 4.$
Moreover, as in the stationary setting, we treat separately the massless and the massive cases.
 
\begin{theorem}[Massless, 3D, Dirac equation]
\label{thm:3-dim-massless}
 Let $d=3,$ $m=0,$ and let Assumption~\ref{ass:HE} be satisfied.  Assume that there exist $\epsilon\in (0,1)$ and two positive constants $C_1$ and $C_2$ such that for any $j=1,2,3,$ 
 \begin{equation}\label{eq:hyp-higher-mzero-3d}
 |\nabla V|\leq \frac{C_1}{|x|^{2-\epsilon}+ |x|^{2+\epsilon}},
 \qquad
 |[\alpha_j,V]|\leq \frac{C_2}{|x|^{1-\epsilon}+ |x|^{1+\epsilon}},
 \end{equation}
 and such that
 \begin{equation*}
 	(c_{\epsilon/2})^2\left[\tfrac{3\sqrt{2}}{4} C_1  + \tfrac{3\sqrt{2}+ 12}{8} C_2 \right] <\tfrac{1}{6},
 	\qquad 
 	(c_{\epsilon/2})^2 \left[ \tfrac{3\sqrt{2}}{4} C_1 + \tfrac{3\sqrt{2}}{8} C_2 \right]<\tfrac{1}{8},
 \end{equation*}
 with $c_{\epsilon/2}$ as defined in~\eqref{eq:ceps}. 
 Then for any solution $\psi$ to~\eqref{eq:dirac-evolution} the following estimate holds 
	\begin{multline}\label{eq:3d-final-mzero}
 	 \delta  \sup_{R>0} \frac{1}{R} \int_0^\infty \int_{|x|\leq R} |\nabla \psi|^2\, dx\, dt
 	 + \frac{1}{2}\int_0^\infty\int_{\R^3} \frac{|\partial_\tau \psi|^2}{|x|}\, dx\, dt
 	 + \delta \sup_{R>0} \frac{1}{R^2} \int_0^\infty\int_{|x|=R} |\psi|^2\, d\sigma \, dt\\
 	 \lesssim_d \| (H_D+V)\psi(0,\cdot) \|_{L^2}^2, 
\end{multline}  
for some $\delta>0.$
 \end{theorem}
 
 The corresponding  result in the massive case reads as follows:
 \begin{theorem}[Massive, 3D, Dirac equation]
\label{thm:3-dim-massive}
 Let $d=3,$ $m>0,$ and let Assumption~\ref{ass:HE} be satisfied.  Assume that there exist $\epsilon\in (0,1)$ and four positive constants $C_1, C_2, C_3$ and $C_4$ such that for any $j=1,2,3,$ 
 \begin{equation}\label{eq:hyp-higher-m-3d} 
 \begin{gathered}
 	 |\nabla V|\leq \frac{C_1}{|x|^{2-\epsilon}+ |x|^{2+\epsilon}},
 \qquad
 |[\alpha_j,V]|\leq \frac{C_2}{|x|^{1-\epsilon}+ |x|^{1+\epsilon}},
 \\
  |(x\cdot \nabla )\{V,\beta\}|\leq \frac{C_3}{|x|^{2-\epsilon}+ |x|^{2+\epsilon}},
 \qquad 
|[\beta,V]|\leq \frac{C_4}{|x|^{2-\epsilon}+ |x|^{2+\epsilon}},
\end{gathered}
 \end{equation}
  and such that
 \begin{equation*}
 	(c_{\epsilon/2})^2 \left[\tfrac{3\sqrt{2}}{4} C_1 + \tfrac{3\sqrt{2}+ 12}{8} C_2 + \tfrac{3\sqrt{2}m}{8} C_4\right] <\tfrac{1}{6},
 	\qquad 
 	(c_{\epsilon/2})^2 \left[\tfrac{3\sqrt{2}}{4} C_1 
 	+ \tfrac{3\sqrt{2}}{8} C_2 
 	+\tfrac{3m}{8} C_3
 	+\tfrac{3\sqrt{2}m}{8} C_4
 	\right]<\tfrac{1}{8},
 \end{equation*}
 with $c_{\epsilon/2}$ as defined in~\eqref{eq:ceps}. 
	 Then for any solution $\psi$ to~\eqref{eq:dirac-evolution} the bound in~\eqref{eq:3d-final-mzero} holds with a smaller $\delta.$
\end{theorem}

Now we are in position to state the alternative results in higher dimensions $d\geq 4.$ We shall start with the result in the massless case.

\begin{theorem}\label{thm:higher-dim-massless}
 Let $d\geq 4,$ $m=0,$ and let Assumption~\ref{ass:HE} be satisfied. Assume that there exist $\epsilon\in(0,1/2)$ and two positive constants $C_1$ and $C_2$ such that for any $j=1,2,\dots, d,$ 
\begin{equation}\label{eq:hyp-higher-mzero}
 	 |\nabla V|\leq \frac{C_1}{|x|^{2-\epsilon}+ |x|^{2+\epsilon}},
 \qquad
 |[\alpha_j,V]|\leq \frac{C_2}{|x|^{1-\epsilon}+ |x|^{1+\epsilon}},
 \end{equation}
    and such that
 \begin{equation*}
 	 \left[\tfrac{3}{4} C_1 + \tfrac{3(d-1)}{16} C_2 \right] c_{\epsilon}+ \tfrac{3}{2} C_2 (c_{\epsilon/2})^2 <\tfrac{d-1}{4d},
 	\qquad 
 	 \left[\tfrac{3}{4} C_1 
 	+ \tfrac{3(d-1)}{16} C_2 
 	 \right]c_\epsilon
 	<\tfrac{(d-1)(d-3)}{8},
 \end{equation*}
 with $c_{\varepsilon},$ $\varepsilon=\epsilon$ or $\varepsilon=\epsilon/2,$ as defined in~\eqref{eq:ceps}. 
 Then for any solution $\psi$ to~\eqref{eq:dirac-evolution} the following estimate holds 
	\begin{multline}\label{eq:higher-final-mzero}
 	 \delta  \sup_{R>0} \frac{1}{R} \int_0^\infty \int_{|x|\leq R} |\nabla \psi|^2\, dx\, dt
 	 + \delta \int_0^\infty \int_{\R^d} \frac{|\psi|^2}{|x|^3}\, dx\, dt\\
 	 + \frac{1}{2}\int_0^\infty\int_{\R^d} \frac{|\partial_\tau \psi|^2}{|x|}\, dx\, dt 	
 	 + \frac{(d-1)}{16} \sup_{R>0} \frac{1}{R^2} \int_0^\infty\int_{|x|=R} |\psi|^2\, d\sigma \, dt
 	 \lesssim_d \| (H_D+V)\psi(0,\cdot) \|_{L^2}^2, 
\end{multline}  
for some $\delta>0.$
\end{theorem}

The corresponding result in the massive case reads as follows:

\begin{theorem}\label{thm:higher-dim-massive}
 Let $d\geq 4,$ $m>0,$ and let Assumption~\ref{ass:HE} be satisfied. Assume that there exist $\epsilon\in (0,1)$ and four positive constants $C_1, C_2, C_3$ and $C_4$ such that for any $j=1,2,\dots, d,$
\begin{equation}\label{eq:hyp-higher-m}
 	 |\nabla V|\leq \frac{C_1}{|x|^{2-\epsilon}+ |x|^{2+\epsilon}},
 \quad
 |[\alpha_j,V]|\leq \frac{C_2}{|x|^{1-\epsilon}+ |x|^{1+\epsilon}},
 \quad
  |(x\cdot \nabla)\{V,\beta\}|\leq \frac{C_3}{|x|^2},
 \quad 
|[\beta,V]|\leq \frac{C_4}{|x|^{2-\epsilon}+ |x|^{2+\epsilon}},
 \end{equation}
  and such that
 \begin{equation*}
 	 \left[\tfrac{3}{4} C_1 + \tfrac{3(d-1)}{16} C_2 + \tfrac{3m}{8}C_4 \right] c_{\epsilon}+ \tfrac{3}{2} C_2 (c_{\epsilon/2})^2 <\tfrac{d-1}{4d},
 	\qquad 
 	 \left[\tfrac{3}{4} C_1 
 	+ \tfrac{3(d-1)}{16} C_2 
 	+\tfrac{3m}{8} C_4 \right]c_\epsilon
 	+\tfrac{3m}{8} C_3 
 	<\tfrac{(d-1)(d-3)}{8},
 \end{equation*}
 with $c_{\varepsilon},$ $\varepsilon=\epsilon$ or $\varepsilon=\epsilon/2,$ as defined in~\eqref{eq:ceps}. 
 Then for any solution $\psi$ to~\eqref{eq:dirac-evolution} the bound in~\eqref{eq:higher-final-mzero} holds with a smaller $\delta.$
\end{theorem}

\begin{remark}\label{rem:control-gradient}
The term $\|(H_D+V)\psi(0,\cdot)\|_{L^2}^2$ appearing in the statements of Theorems~\ref{thm:3-dim-massless}--\ref{thm:higher-dim-massive} can be replaced by $\|\nabla\psi(0,\cdot)\|_{L^2}+m^2\|\psi(0,\cdot)\|_{L^2}$, that is, the smoothing estimates depend only on the $H^1$-norm of the initial datum $\psi(0,\cdot)$, see Remark~\ref{rem:almost-conservation} below for more details.
\end{remark}

\subsubsection*{Structure of the paper}
In the preliminary Section~\ref{sec:preliminary}, we
collect the notation and basic results used throughout the paper. 

The derivation of the new
relativistic virial identities for the stationary equation is presented in
Section~\ref{sec:spectral}, together with the proofs of the spectral stability results
stated in Theorems~\ref{thm:spectral-stability} and~\ref{thm:spectral-stability2}.

In Section~\ref{sec:local-smoothing}, these virial identities are further extended to the
evolutionary setting and applied to establish the local smoothing estimates of
Theorems~\ref{thm:3-dim-massless}–\ref{thm:higher-dim-massive}.

Finally, in Appendix~\ref{appendix:1}, we prove some results from
Section~\ref{sec:preliminary}.

\section{Preliminaries}\label{sec:preliminary}
We collect here several properties that will be used consistently throughout the paper. When a proof is omitted, the result is straightforward.
\begin{itemize}
	\item From the anti-commutation relations~\eqref{eq:commutation-rel-dirac-matr} satisfied by the Dirac matrices, for any $\psi\in \mathcal{D}(H_D)$ one has
	\begin{equation}\label{eq:free-dirac-identities}
		\|H_D \psi\|_{L^2}=\|\nabla \psi\|_{L^2},
		\qquad 
		\|H_D \psi\|_{L^2}^2=\|\nabla \psi\|_{L^2}^2 + m^2 \|\psi\|_{L^2}^2.
	\end{equation}	
	\item The dilation operator $G$ defined in~\eqref{eq:generator-dilations} is symmetric, thus $iG$ is anti-symmetric. Moreover, one has 
	\begin{equation}\label{eq:iG}
		iG= \frac{d}{2} + x\cdot \nabla.
	\end{equation}
	\item The commutator operator $[iG,V]$ can be explicitly written as
	\begin{equation}\label{eq:comm-iGV}
		[iG,V]=(x\cdot \nabla) V.
	\end{equation}
	\item For any \emph{radial function} $\varphi,$ one has
\begin{equation}\label{eq:radial-comp-iG}
	[H_0,\varphi]=-\varphi''-\tfrac{(d-1)}{r} \varphi' -2\varphi' \partial_r.
\end{equation}
	\item Using~\eqref{eq:radial-comp-iG}, it is easy to see that for $iG_\phi$ as defined in~\eqref{eq:iGphi} and $\phi$ a radial function one has
	\begin{equation}\label{eq:iGphi-computed}
	iG_\phi=-\tfrac{1}{4}\left(\phi''+\tfrac{(d-1)}{r} \phi'\right) -\tfrac{1}{2}\phi' \partial_r
	\end{equation}
	\item
	The commutator operator $[iG_\phi,V]$ can be explicitly written as
	\begin{equation}\label{eq:comm-iGfV}
		[iG_\phi,V]=\tfrac{1}{2} \phi' \partial_r V.
	\end{equation}
	\item The commutator operator $[H_D,V]$ is anti-symmetric. Moreover it can be explicitly written as
	\begin{equation}\label{eq:comm-HDV}
		[H_D,V]= -i [(\alpha \cdot \nabla) V] -i[\alpha_j, V]\partial_j + m[\beta, V].
	\end{equation}
	\item We will be needing the \emph{classical Hardy inequality}
	\begin{equation}\label{eq:classical-Hardy}
		\int_{\R^d} \frac{|\psi(x)|^2}{|x|^2}\, dx \leq C_{\text{H}}^2 \int_{\R^d} |\nabla \psi|^2(x)\, dx, \qquad \forall \, \psi\in C^\infty_0(\R^d), \quad d\geq 3,
	\end{equation}
	where the best constant $C_{\text{H}}^2$ is given by $C_{\text{H}}^2:=4/(d-2)^2.$
\item The Morrey-Campanato norm is defined as
\begin{equation}\label{eq: Morrey-C}
    \vertiii{u}^2 := \sup_{R>0} \frac{1}{R} \int_{|x|\leq R} |u|^2\, dx.
\end{equation}
We will make consistent use of the bound
\begin{equation}\label{eq:weighted-to-Morrey}
    \left(\int_{\R^d} \frac{|u|^2}{(|x|^{1/2-\delta} + |x|^{1/2+\delta})^2}\,dx\right)^{1/2} \leq c_\delta \vertiii{u},
    \quad\text{for}\ \delta\in (0,1/2),
\end{equation}
where $c_\delta$ is given in~\eqref{eq:ceps}. The proof is deferred to Appendix~\ref{appendix:1}.

\item We also introduce the spherical norm
\begin{equation*}
    \|u\|_{R}^2 := \sup_{R>0} \frac{1}{R^2} \int_{|x|=R} |u|^2\,d\sigma,
\end{equation*}
together with the analogous bound
\begin{equation}\label{eq:weighted-to-spherical}
    \left(\int_{\R^d} \frac{|u|^2}{(|x|^{3/2-\delta} + |x|^{3/2+\delta})^2}\,dx\right)^{1/2} \leq \frac{c_\delta}{\sqrt{2}} \|u\|_{R},
    \quad\text{for}\ \delta\in (0,1/2),
\end{equation}
where $c_\delta$ defined in~\eqref{eq:ceps}. The proof is again deferred to Appendix~\ref{appendix:1}.
\end{itemize}

\section{On the point spectrum of the perturbed Dirac operator}
\label{sec:spectral}
Throughout this section, we assume Assumption~\ref{ass:HS} and consider the following eigenvalue problem
\begin{equation}\label{eq:eigenvalue-pb}
	H\psi=\lambda \psi,
	\qquad \lambda \in \R,
\end{equation}
associated to the self-adjoint Hamiltonian $H=H_D+V,$ with $H_D$ being the free Dirac operator defined in~\eqref{eq:Dirac}.

Our main purpose is to find suitable smallness conditions on $V$ such that solutions to~\eqref{eq:eigenvalue-pb} cannot exist. In order to do that we will need the following easy integral identity.
\begin{lemma}[Virial identity]\label{lemma:virial-identity}
	Let $\psi\in \mathcal{D}(H)$ be a solution to~\eqref{eq:eigenvalue-pb}, then one has 
	\begin{equation}\label{eq:virial-identity}
		\frac{1}{2}\langle \psi, [H,L] \psi \rangle=0,		
	\end{equation}
	where $L$ is the second-order, anti-symmetric operator defined in~\eqref{eq:L}.
\end{lemma}
\begin{proof}
	The proof is obtained through an easy application of the so-called method of multipliers: multiplying (in $L^2$) equation~\eqref{eq:eigenvalue-pb} by the test function $L\psi,$ with $L$ as defined in~\eqref{eq:L} and then considering the real parts of the resulting identities yields
	\begin{equation}\label{eq:v1}
		\Re \langle H\psi, L\psi \rangle = \lambda \Re \langle \psi, L\psi \rangle.
	\end{equation}
	Since $L$ is anti-symmetric, the expectation value $\langle \psi, L\psi \rangle$ is purely imaginary, thus the right-hand-side of~\eqref{eq:v1} is zero. Identity~\eqref{eq:virial-identity} then follows simply observing that $\Re \langle H\psi, L\psi \rangle= \tfrac{1}{2}\langle \psi, [H,L] \psi \rangle.$
\end{proof}

\begin{remark}
To make the proof of Lemma~\ref{lemma:virial-identity} fully rigorous, one has to ensure that the test function $v=L\psi$ indeed belongs to $L^2(\R^d).$ In general, assumption $\psi\in H^1$ is not sufficient to guarantee this, since $v$ is possibly unbounded and the operator $L$ is second order. A rigorous proof then shoud begin by approximating the solution $\psi$ with a sequence of compactly supported functions $\psi_R$ vanishing near the origin that satisfy a related problem. Then, one further observes that hypothesis~\ref{ass:HS} on the potential allows to prove that each $\psi_R$ enjoys higher regularity, specifically $\psi_R\in H^2(\R^d).$ The desired identity can thus be rigorously  established for $\psi_R$ and identity~\eqref{eq:virial-identity} in Lemma~\ref{lemma:virial-identity} eventually follows by passing to the limit as $R \to \infty.$ For a detailed presentation of this regularization argument, see~\cite[Lemma 3.1]{ArrizabalagaCossettiMorales2025}.
\end{remark}

 We shall now expand further the commutator $[H,L].$ Due to expression~\eqref{eq:spectral}, we will only need to consider the part involving the potential, namely $[V,L].$ We show the following result.
 \begin{lemma}\label{lemma:potential-rewrite}
 \begin{equation}\label{eq:v2}
 	\frac{1}{2}\langle \psi, [V,L]\psi \rangle
 	=\Re \langle [H_D, V] \psi, iG \psi \rangle 
 	-\Re \langle [iG, V]\psi, H_D \psi \rangle.
 \end{equation}
 \begin{proof}
 	Writing explicitly the operator $L$ as defined in~\eqref{eq:L}, one has 
 	\begin{equation*}
 	\begin{split}
 		\frac{1}{2}\langle \psi, [V,L]\psi \rangle&= \Re \langle V\psi, L\psi \rangle\\
 		&=\Re \langle V\psi, H_D iG\psi \rangle + \Re \langle V\psi, iG H_D\psi \rangle\\
 		&=\Re \langle H_D V\psi, iG\psi \rangle - \Re \langle iG V\psi, H_D\psi \rangle\\
 		&=\Re \langle V H_D\psi, iG\psi \rangle + \Re \langle [H_D,V]\psi, iG\psi \rangle -\Re \langle iG V\psi, H_D\psi \rangle\\
 		&=\Re \langle V H_D\psi, iG\psi \rangle + \Re \langle [H_D,V]\psi, iG\psi \rangle -\Re \langle V iG \psi, H_D\psi \rangle -\Re \langle [iG,V] \psi, H_D\psi \rangle \\
 		&=\Re \langle [H_D, V] \psi, iG \psi \rangle 
 	-\Re \langle [iG, V]\psi, H_D \psi \rangle,
 	\qedhere
 	\end{split}
 	\end{equation*} 	
 \end{proof}
 \end{lemma}

\begin{remark}\label{rem:Vphi-alternative}
In the proof of identity~\eqref{eq:v2} contained in Lemma~\ref{lemma:potential-rewrite} we have used only that $V$ and $H_D$ are symmetric and that $L$ and $iG$ are skew-symmetric. Thus identity~\eqref{eq:v2} also holds replacing $L$ with $L_\phi$ and $iG$ with $iG_\phi,$ where $L_\phi$ and $iG_\phi$ are defined in~\eqref{eq:L2} and~\eqref{eq:iGphi}, respectively.
\end{remark}

Now we are in position to prove our results. We shall prove first the result in the massive case, namely Theorem~\ref{thm:spectral-stability2}, then the massless counterpart Theorem~\ref{thm:spectral-stability} will follow as an easy consequence. 
\begin{proof}[Proof of Theorem~\ref{thm:spectral-stability2}]
We consider as a starting point identity~\eqref{eq:v2} above. Since $[H_D,V]$ is anti-symmetric, then $\langle [H_D,V] \psi, \psi \rangle$ is purely immaginary, therefore, in view of~\eqref{eq:iG}, identity~\eqref{eq:v2} reduces to
\begin{equation*}
\tfrac{1}{2}\langle \psi, [V,L]\psi \rangle
 	=\Re \langle [H_D, V] \psi, x\cdot \nabla \psi \rangle 
 	-\Re \langle [iG, V]\psi, H_D \psi \rangle.
\end{equation*}
Now, using the explicit expressions for the commutators  $[H_D, V]$ and $[iG, V]$ given in~\eqref{eq:comm-HDV} and in~\eqref{eq:comm-iGV}, respectively, one has 
\begin{multline*}
	\frac{1}{2}\langle \psi, [V,L]\psi \rangle
	\\=\Re \langle -i [(\alpha \cdot \nabla) V] \psi, x\cdot \nabla \psi \rangle+
	\Re \langle-i[\alpha_j, V]\partial_j \psi, x\cdot \nabla \psi \rangle
	+m \Re \langle [\beta,V] \psi, x\cdot \nabla \psi \rangle 
 	-\Re \langle [(x\cdot \nabla) V]\psi, H_D \psi \rangle.
\end{multline*} 
To conclude, it remains to estimate the following five terms:\begin{equation*}
\begin{gathered}
	I_1:=\Re \langle -i[(\alpha \cdot \nabla) V] \psi, x\cdot \nabla \psi \rangle,
	\quad
	I_2:=\Re \langle -i[\alpha_j,V] \partial_j\psi, x\cdot \nabla \psi \rangle,
	\quad
	I_3:=m \Re \langle [\beta,V] \psi, x\cdot \nabla \psi \rangle,\\
	I_4:=-\Re \langle [(x\cdot \nabla) V]\psi,-i \alpha\cdot \nabla \psi \rangle,\quad
	I_5:=-m\Re \langle [(x\cdot \nabla) V]\psi, \beta \psi \rangle.
\end{gathered}
\end{equation*}
Using Cauchy-Schwarz, the Hardy inequality and hypotheses~\eqref{eq:ifin2} one has
\begin{equation*}
	|I_1| + |I_4|\leq 2 \||x||\nabla V| |\psi|\|\|\nabla \psi\|
	\leq \frac{4C_1}{d-2} \|\nabla \psi\|^2.
\end{equation*}
As for $I_2,$ as above, one has
\begin{equation*}
	|I_2| \leq \||x||[\alpha_j, V]| |\partial_j\psi|\|\|\nabla \psi\|
	\leq C_2 \|\nabla \psi\|^2.
\end{equation*}
As for $I_3$ one has
\begin{equation*}
	|I_3| \leq  m\||x||[\beta, V]| |\psi|\|\|\nabla \psi\|
	\leq \frac{2mC_4}{d-2} \|\nabla \psi\|^2,
\end{equation*}
We finally turn to the estimate of $I_5$. Since the multiplication operator $(x\cdot\nabla) V$ is symmetric (as $V$ is Hermitian) and $\beta$ is symmetric, 
$I_5$ can be estimated in terms of the anticommutator of these two operators. Moreover, since $\beta$ is a constant matrix, one has 
$\{(x\cdot\nabla) V,\beta\}=(x\cdot\nabla)\big[\{V,\beta\}\big]$. Hence, using Hardy's inequality one has
\begin{equation*}
|I_5|=\left|\frac{m}{2}\langle \psi, (x\cdot \nabla) [\{V,\beta\}]\psi \rangle\right|
\leq  \frac{2 m C_3}{(d-2)^2}\|\nabla\psi\|^2.
\end{equation*} 
Now, using in order~\eqref{eq:virial-identity},~\eqref{eq:spectral} and the last bounds, one has 
\begin{equation}\label{eq:estimate-nabla-psi}
\begin{split}
	\|\nabla \psi\|^2 &\leq |I_1|+ |I_2|+|I_3|+|I_4|+|I_5|\\
	&\leq \Big(\frac{4C_1}{d-2} + C_2 + \frac{2mC_4}{d-2} + \frac{2mC_3}{(d-2)^2} \Big)\|\nabla \psi\|^2.
\end{split}
\end{equation}
This, due to hypothesis~\eqref{eq:tesi1}, gives the thesis. 
\end{proof}

\begin{proof}[Proof of Theorem~\ref{thm:spectral-stability}]
The proof of this result simply follows from the previous one observing that in this case $I_3$ and $I_5$ are zero.
\end{proof}

\section{Local smoothing for the perturbed Dirac equation}
\label{sec:local-smoothing}
Throughout this section, we assume Assumption~\ref{ass:HE}
and we prove the local smoothing estimates for the perturbed Dirac operator given in the introduction.
Unlike in Section~\ref{sec:spectral}, the functions considered here depend on both the time and space variables, i.e.\ $\psi = \psi(t,x)$. We omit this dependence when it is clear from the context, and make it explicit when it is relevant.

In order to justify the algebraic manipulations introduced in this section, we approximate the initial data $f\in H^1(\R^d)$ corresponding to our solution $\psi=e^{itH}f,$ by a sequence $f_j\in C^\infty_c(\R^d),$ and note that the corresponding solutions $\psi_j:=e^{itH}f_j\in C(\R; H^{3/2}(\R^d).$ For readability, we drop the subscript $j$ in what follows.

\medskip
The starting point for the proofs of Theorem~\ref{thm:3-dim-massless} - Theorem~\ref{thm:higher-dim-massive} is the following virial identity for the evolution equation~\eqref{eq:dirac-evolution}.
\begin{lemma}[Time-dependent virial identity]
\label{lemma:gen-id}
	Let $\psi$ be a solution to~\eqref{eq:dirac-evolution}, then one has 
\begin{equation}\label{eq:gen-id}
	\frac{d}{dt} \Im \langle H_D \psi, iG_\phi \psi \rangle 
	= -\frac{1}{2}\int_{\R^d} \nabla \psi D^2 \phi  \overline{\nabla\psi}\, dx+\frac{1}{8} \int_{\R^d} \Delta^2 \phi |\psi|^2\, dx -\frac{1}{2} \langle \psi, [V, L_\phi]\psi \rangle,
\end{equation}	
where $iG_\phi$ is defined in~\eqref{eq:iGphi} and $L_\phi$ is defined in~\eqref{eq:L2}.
\begin{proof}
In order to obtain~\eqref{eq:gen-id} one multiplies equation~\eqref{eq:dirac-evolution} by the test function $L_\phi \psi,$ with $L_\phi$ as defined in~\eqref{eq:L2}, and then takes the real part of the resulting identity. This gives
\begin{equation}\label{eq:preliminary}
	\Re \langle i\partial_t \psi, L_\phi \psi \rangle
	= \Re \langle H\psi,  L_\phi \psi  \rangle.
\end{equation}
Let us start from the left-hand side. Using that $L_\phi$ is anti-symmetric and time-independent, one has
\begin{equation*}
\begin{split}
	\Re \langle i\partial_t \psi, L_\phi \psi \rangle
	&= - \Im \langle \partial_t \psi, L_\phi \psi \rangle\\
	&= \frac{i}{2} \big(\langle \partial_t \psi,L_\phi \psi \rangle - \langle L_\phi \psi, \partial_t \psi \rangle \big)\\
	&=\frac{i}{2} \frac{d}{dt}\langle \psi, L_\phi \psi \rangle.
\end{split}
\end{equation*}
Computing explicitly $\langle \psi, L_\phi \psi \rangle$ we have
\begin{equation*}
\begin{split}
	\langle \psi, L_\phi \psi \rangle
	&=\langle \psi, \{H_D,iG_\phi\} \psi \rangle\\
	&=\langle \psi, H_D  iG_\phi \psi \rangle + \langle \psi, iG_\phi H_D \psi \rangle\\
	&= \langle H_D \psi, iG_\phi \psi \rangle - \langle iG_\phi \psi, H_D \psi \rangle\\
	&=2i \Im \langle H_D \psi, iG_\phi \psi \rangle.   
	\end{split}
\end{equation*}
Now we consider the right-hand side of \eqref{eq:preliminary}. Using the identity~\eqref{eq:morawetzid} and writing explicitly the double commutator one has
\begin{equation*}
\begin{split}
\Re \langle H \psi, L_\phi \psi \rangle
&= \frac{1}{2} \langle \psi, [H,L_\phi]\psi \rangle\\
&= -\frac{1}{8} \langle \psi, [H_0, [H_0, \phi]]\psi \rangle +\frac{1}{2} \langle \psi, [V, L_\phi]\psi \rangle\\
&=\frac{1}{2}\int_{\R^d} \nabla \psi D^2 \phi  \overline{\nabla\psi}-\frac{1}{8} \int_{\R^d} \Delta^2 \phi |\psi|^2\, dx +\frac{1}{2} \langle \psi, [V, L_\phi]\psi \rangle. 
\end{split}
\end{equation*}
Putting those identities in~\eqref{eq:preliminary} and multiplying the resulting identity by $-1$ one gets~\eqref{eq:gen-id}.
\end{proof}
\end{lemma}
\begin{remark}
Notice that the previous computations are rigorous due to the assumption $\psi\in H^{3/2}.$ Indeed, if $\psi\in H^{3/2},$ then $L_\phi \psi \in H^{-1/2},$ while $H_D \psi \in H^{1/2},$ and $V\psi \in H^{1/2}.$ The latter follows from the Kato smallness~\eqref{eq:smallness-potentia} in assumption~\ref{ass:HE} and the pointwise bound on $\nabla V.$ Therefore, the brackets in~\eqref{eq:preliminary} are understood as the duality pairing between $H^{1/2}$ and $H^{-1/2}.$
\end{remark}

As already mentioned, identity~\eqref{eq:gen-id} provides the starting point for the proofs of Theorems~\ref{thm:3-dim-massless}–\ref{thm:higher-dim-massive}. To carry this out, an appropriate multiplier $\phi$ must be selected. In the next section, we introduce this multiplier explicitly and describe its main properties. 

\subsection{The choice of the multiplier}
Let $r:=|x|,$ we define 
\begin{equation}\label{eq:multiplier-choice}
	\phi(r)=\phi_{\text{M}}(r) + \phi_{\text{ls}}(r),
\end{equation}
where $\phi_{\text{M}}(r):=r$ is the so-called \emph{Morawetz} multiplier, while $\phi_{\text{ls}}(r)$ is the \emph{local smoothing} multiplier defined as follows: for any $R>0$ we take $\phi_{\text{ls}}(r):=\int_0^r \varphi_{\text{ls}}'(s)\, ds,$ with
\begin{equation}\label{eq:phi1-ls}
	\phi_{\text{ls}}'(r)=
	\begin{cases}
		\dfrac{(d-1)}{2d} \dfrac{r}{R},  &\text{for}\ r\leq R\\[8pt]
		\dfrac{1}{2}- \dfrac{R^{d-1}}{2d r^{d-1}} &\text{for}\ r>R.
	\end{cases}
\end{equation}
As customarily, the local smoothing multiplier $\phi_{\text{ls}}$ is defined as a function that has the same behavior of the virial multiplier $r^2$ inside a ball of radius $R$ centered at zero and it behaves as the Morawetz multiplier $\phi_{\text{M}}$ outside this ball.

\medskip
Once the multiplier is selected, we shall plug our choice~\eqref{eq:multiplier-choice} in the virial identity~\eqref{eq:gen-id} obtained in Lemma~\ref{lemma:gen-id}. Then our results will be obtained after carefully estimating each term of the resulting identity. 

For the sake of clarity, we split these computations into three parts: the estimate of the \emph{kinetic term} on the right-hand side of~\eqref{eq:gen-id}, the estimate of the \emph{potential term} on the right-hand side of~\eqref{eq:gen-id}, and finally the treatment of the \emph{time-derivative term} on the left-hand side of~\eqref{eq:gen-id}.

\subsection{Kinetic term}
This section is devoted to establish a suitable lower bound for the following quantity
\begin{equation}\label{eq:kinetic}
	K:=\frac{1}{2}\int_{\R^d} \nabla \psi D^2 \phi \overline{\nabla \psi}\, dx-\frac{1}{8} \int_{\R^d} \Delta^2 \phi |\psi|^2\, dx,
\end{equation}
with $\phi$ as defined in~\eqref{eq:multiplier-choice}, which appears in the right-hand-side of~\eqref{eq:gen-id}. The main result of this section is the following:
\begin{proposition}\label{prop:K-lower-bound}
Let $K$ be the kinetic term defined in~\eqref{eq:kinetic}. Then one has
\begin{equation}\label{eq:K-lower-bound}
	K\geq 
	\frac{1}{2} \int_{\R^d} \frac{|\partial_\tau \psi|^2}{|x|}\, dx 
	+ \frac{(d-1)}{4d} \frac{1}{R}\int_{|x|\leq R}|\nabla \psi|^2\, dx
	+ \frac{(d-1)(d-3)}{8} \int_{\R^d} \frac{|\psi|^2}{|x|^3}\, dx
	+ \frac{(d-1)}{16}\frac{1}{R^2} \int_{|x|=R} |\psi|^2\, d\sigma(x).
\end{equation}
\end{proposition}
\begin{remark}
	 In the lower bound~\eqref{eq:K-lower-bound}, the third term on the right-hand side vanishes when $d=3.$ Consequently, in three dimensions there is less positivity available to control the terms arising from the perturbation.
\end{remark}
To prove Proposition~\ref{prop:K-lower-bound}, we need a more explicit expression for $\nabla \psi D^2 \phi \overline{\nabla \psi}$ and $\Delta^2 \phi.$
Recalling~\eqref{eq:multiplier-choice}, the following results will be useful for this purpose. 
\begin{lemma}\label{lemma:after-multiplier-choice-M}
Let $\phi_{\text{M}}$ be the Morawetz multiplier, namely $\phi_{\text{M}}(r)=r.$ Then one has 
\begin{equation}
	\label{eq:plug-Hessian-bilaplacian-M}
	\nabla \psi D^2\phi_{\text{M}} \overline{\nabla \psi}=\frac{|\partial_\tau \psi|^2}{r},
	\qquad 
	\Delta^2 \phi_{\text{M}}=
	\begin{cases}
	-8\pi \delta_{x=0} \qquad &d=3,\\[6pt]
	-\frac{(d-1)(d-3)}{r^3}&d\geq 4.
	\end{cases}
\end{equation}	
\end{lemma}
\begin{lemma}\label{lemma:after-multiplier-choice-ls}
Let $\phi_{\text{ls}}$ be the local smoothing multiplyier defined as in~\eqref{eq:phi1-ls}. Then one has 
\begin{align}
	\label{eq:plug-Hessian-ls}
	\nabla \psi D^2\phi_{\text{ls}} \overline{\nabla \psi}&=
	\tfrac{(d-1)}{2d} \tfrac{1}{R}|\nabla \psi|^2 \chi_{\{|x|\leq R\}} 
 + \Big[ \tfrac{(d-1)}{2d} \tfrac{R^{d-1}}{r^d} |\partial_r \psi|^2
+ \tfrac{1}{2d} \tfrac{R^{d-1}}{r^d} \Big(\tfrac{dr^{d-1}}{R^{d-1}} - 1\Big)|\partial_\tau \psi|^2
 \Big]\chi_{\{|x|> R\}},\\
 \label{eq:plug-bilaplacian-ls}
	\Delta^2 \phi_{\text{ls}}&=-\tfrac{(d-1)(d-3)}{2r^3}\chi_{[R,\infty)}-\tfrac{(d-1)}{2}\tfrac{1}{R^2}\delta_{|x|=R}.
\end{align}	
\end{lemma}

\begin{remark}\label{rem:lower-bound-hessian}
An immediate consequence of identity~\eqref{eq:plug-Hessian-ls} above is the following useful estimate
\begin{equation}\label{eq:rem:lower-bound-hessian}
	\nabla \psi D^2\phi_{\text{ls}} \overline{\nabla \psi}
	\geq \tfrac{(d-1)}{2d} \tfrac{1}{R}|\nabla \psi|^2 \chi_{\{|x|\leq R\}} 
	+ \tfrac{(d-1)}{2d} \tfrac{R^{d-1}}{r^d}|\nabla \psi|^2 \chi_{\{|x|>R\}}
	\geq \tfrac{(d-1)}{2d} \tfrac{1}{R}|\nabla \psi|^2 \chi_{\{|x|\leq R\}},
\end{equation}
where in the last inequality we have just discarded the second term being positive.
\end{remark}

We shall start with the proof of Lemma~\ref{lemma:after-multiplier-choice-M}.
\begin{proof}[Proof of Lemma~\ref{lemma:after-multiplier-choice-M}]
As a starting point we observe that, for any \emph{radial function} $\varphi,$ one has
\begin{equation}\label{eq:radial-computations}
	\nabla \psi D^2\varphi \overline{\nabla \psi}= \varphi'' |\partial_r \psi|^2 + \tfrac{\varphi'}{r} |\partial_\tau \psi|^2,
	\qquad	
	\Delta \varphi=\varphi'' + \tfrac{(d-1)}{r} \varphi'.
\end{equation}
Plugging $\phi_{\text{M}}$ in the first identity in~\eqref{eq:radial-computations}, then the first in~\eqref{eq:plug-Hessian-bilaplacian-M} is immediately obtained.

Now we prove the second identity in~\eqref{eq:plug-Hessian-bilaplacian-M}. Using the second identity in~\eqref{eq:radial-computations} one easily has $\Delta \phi_{\text{M}}=\tfrac{(d-1)}{r}.$
Due to the presence of the function $1/r$ in the previous expression for $\Delta \phi_{\text{M}},$ which represents the fundamental solution of the Laplacian in dimension $d=3,$ the cases $d=3$ and $d\geq 4$ are slightly different. An easy computation shows that
\begin{equation*}
	\Delta^2 \phi_{\text{M}}(r)=
	\begin{cases}
		-8\pi \delta_{x=0} \qquad &d=3 \\[6pt]
		-\tfrac{(d-1)(d-3)}{r^3}  &d\geq 4,
	\end{cases}
\end{equation*}
which is the thesis.
\end{proof}

We continue with the proof of Lemma~\ref{lemma:after-multiplier-choice-ls}.

\begin{proof}[Proof of Lemma~\ref{lemma:after-multiplier-choice-ls}]
The proof is similar to the proof of Lemma~\ref{lemma:after-multiplier-choice-M}. One just needs to use that 
\begin{equation}\label{eq:second-der-ls}
	\phi_{\text{ls}}''(r)=
	\begin{cases}
		\dfrac{(d-1)}{2d} \dfrac{1}{R}, \quad &r\leq R \vspace{0.2cm}\\
		 \dfrac{(d-1)}{2d} \dfrac{R^{d-1}}{r^{d}} &r>R.
	\end{cases}
\end{equation}
Then~\eqref{eq:plug-Hessian-ls} and~\eqref{eq:plug-bilaplacian-ls} are easily obtained. 
\end{proof}

We are finally ready to prove Proposition~\ref{prop:K-lower-bound}
\begin{proof}[Proof of Proposition~\ref{prop:K-lower-bound}]
	The lower bound~\eqref{eq:K-lower-bound} is simply obtained from the identities contained in Lemma~\ref{lemma:after-multiplier-choice-M} and Lemma~\ref{lemma:after-multiplier-choice-ls} and the estimate in Remark~\ref{rem:lower-bound-hessian}.
\end{proof}

\subsection{Potential term}
This section provide an estimate for the \emph{potential term}
\begin{equation}\label{eq:P-definition}
	P:= \frac{1}{2} \langle \psi, [V, L_\phi]\psi \rangle,
\end{equation}
which appears in the right-hand-side of~\eqref{eq:gen-id}, with $\phi$ the multiplier introduced in~\eqref{eq:multiplier-choice}.

From Lemma~\ref{lemma:potential-rewrite} and Remark~\ref{rem:Vphi-alternative}, one has 
\begin{equation*}
 	P
 	=\Re \langle [H_D, V] \psi, iG_\phi \psi \rangle 
 	-\Re \langle [iG_\phi, V]\psi, H_D \psi \rangle.
 	\\
\end{equation*}
Let us analyse both term on the right-hand-side separately. About the first, thanks to \eqref{eq:comm-HDV}  and  \eqref{eq:radial-comp-iG} one has 
\[
\begin{split}
\Re \langle [H_D, V] \psi, iG_\phi \psi \rangle
=&
\Re\langle 
-i [(\alpha \cdot \nabla) V]\psi 
-i[\alpha_j, V]\partial_j\psi 
+ m[\beta, V]\psi,
-\tfrac{1}{4}\left(\phi''+\tfrac{(d-1)}{r} \phi'\right)\psi 
-\tfrac{1}{2}\phi' \partial_r\psi
\rangle
\\
=&
-\tfrac{1}{2}\Re\langle 
-i [(\alpha \cdot \nabla) V]\psi,
\phi' \partial_r\psi
\rangle
-\tfrac{1}{4}\Re\langle 
-i[\alpha_j, V]\partial_j\psi,
\left(\phi''+\tfrac{(d-1)}{r} \phi'\right)\psi 
\rangle\\
&
-\tfrac{1}{2}\Re\langle 
-i[\alpha_j, V]\partial_j\psi,
\phi' \partial_r\psi
\rangle
-\tfrac{m}{2}\Re\langle 
[\beta, V]\psi,
\phi' \partial_r\psi
\rangle.
\end{split}
\]
where in the last equality we have used the fact that both operators $-i[(\alpha\cdot\nabla)V]$ and $[\beta,V]$ are skew-symmetric and that they commute with the symmetric operator $\left(\phi''+\tfrac{(d-1)}{r} \phi'\right)$.

About the second, thanks to \eqref{eq:comm-iGfV} one has
\[
\begin{split}
-\Re \langle [iG_\phi, V]\psi, H_D \psi \rangle&=
\tfrac{1}{2}\Re \langle \left(\phi'\partial_r V\right)\psi, -i\alpha\cdot\nabla \psi \rangle
 	 	+\tfrac{m}{2}\Re \langle \left(\phi'\partial_rV\right)\psi, \beta \psi \rangle
 	 	\\
 	&=\tfrac{1}{2}\Re \langle \left(\phi'\partial_r V\right)\psi, -i\alpha\cdot\nabla \psi \rangle
 	 	+\tfrac{m}{4} \langle \left(\phi'\partial_r\{V,\beta\}\right)\psi, \psi \rangle.
\end{split}
\]
where in the last equality we have used that both operator $\phi'\partial_r V$ and $\beta$ are symmetric and that $\beta$ is a constant matrix.

Summing up we have
\begin{equation}
 \begin{split}\label{eq:v-phi}
 	P
 	=&
 	-\tfrac{1}{2}\Re \langle -i [(\alpha \cdot \nabla) V] \psi,\phi'\partial_r \psi\rangle 
 	-
	\tfrac{1}{4}\Re \langle-i[\alpha_j, V]\partial_j \psi, (\phi''-\tfrac{d-1}{r}\phi') \psi\rangle
	-\tfrac{1}{2}	\Re \langle-i[\alpha_j, V]\partial_j \psi, \phi'\partial_r \psi\rangle\\
	&+\tfrac{1}{2}\Re \langle \left(\phi'\partial_r V\right)\psi, -i\alpha\cdot\nabla \psi \rangle
	+m\left(-\tfrac{1}{2} \Re \langle [\beta,V] \psi, \phi'\partial_r \psi \rangle 
 	 	+\tfrac{1}{4} \langle \left(\phi'\partial_r\{V,\beta\}\right)\psi, \psi \rangle\right).
 \end{split}
 \end{equation}

As a first step, we want to write more explicitly the terms appearing in this equation
\begin{lemma}
Let $P$ be the potential term defined in \eqref{eq:P-definition} with $\phi$ the multiplier defined in \eqref{eq:multiplier-choice} where $\phi_{\text{M}}(r)=r$ is the Morawetz multiplier and $\phi_{\text{ls}}(r)$ is the local smoothing multiplier defined in~\eqref{eq:phi1-ls}.
 Define
	\begin{equation}\label{eq:I1-I2-I3}
		I_1:= \int_{\R^d} |\nabla \psi| |\nabla V| |\psi|\, dx,
		\quad 
		I_2:=\int_{\R^d} |[\alpha_j, V]| |\nabla \psi| \tfrac{|\psi|}{|x|}\, dx,
		\quad 
		I_3:=\int_{\R^d} |[\alpha_j, V]| |\nabla \psi|^2\, dx,
	\end{equation}
	and in the case that the mass $m>0$ define also   
	\begin{equation}\label{eq:Im1-Im2}
	I_{m,1}:= \int_{\R^d} |[\beta, V]| |\psi| |\nabla \psi|\, dx.
	\quad
	I_{m,2}:=\int_{\R^d} |\partial_r \{V,\beta\}| |\psi|^2 \, dx,
		\quad 
	\end{equation} 
Then
\begin{equation}\label{eq:P-first-estimate}
	|P|\leq 
	\frac{3}{2}I_1+\frac{3(d-1)}{8} I_2+\frac{3}{4}I_3+m\left(
	\frac{3}{4} I_{m,1}+\frac{3}{8} I_{m,2}\right).
\end{equation}
\end{lemma}

\begin{proof}
The proof proceeds by direct computation of the derivatives of $\phi$. Indeed, by \eqref{eq:second-der-ls} one has
\begin{equation}\label{eq:estimate-phiprime-phipprime}
\begin{split}
|\phi'|&=
\left(1+\frac{(d-1)}{2d} \frac{r}{R}\right)
\chi_{\{|x|\leq R\}}+
\left(\frac{3}{2}- \frac{R^{d-1}}{2d r^{d-1}}\right)
\chi_{\{|x|> R\}}
\leq
\frac{3d-1}{2d}\chi_{\{|x|> R\}} 
+\frac{3}{2}\chi_{\{|x|> R\}} \leq \frac{3}{2},
\\
\left|\phi''+\tfrac{d-1}{r}\phi'\right|&=
\left(
\frac{d-1}{2dR}
+\frac{d-1}{r}\left(1+\frac{(d-1)}{2d} \frac{r}{R}\right)
\right)
\chi_{\{|x|\leq R\}}+
\left(\frac{d-1}{2d} \frac{R^{d-1}}{r^{d}}+\frac{d-1}{r}\left( 
\frac{3}{2}- \frac{R^{d-1}}{2d r^{d-1}}
\right)
\right)
\chi_{\{|x|> R\}}\\
&
= \left(\frac{d-1}{r}
+\frac{d-1}{2R}
\right)\chi_{\{|x|\leq R\}}+
\left(
\frac{3(d-1)}{2r}\right)\chi_{\{|x|\leq R\}}
\\
&\leq 
\frac{3(d-1)}{r}\chi_{\{|x|\leq R\}}
+ \frac{3(d-1)}{r}\chi_{\{|x|> R\}}
\leq \frac{3(d-1)}{r}.
\end{split}
\end{equation}
Substituting these estimates into \eqref{eq:v-phi} yields \eqref{eq:P-first-estimate}.
\end{proof} 

We conclude this section with a final estimate for the potential term $P$ defined in~\eqref{eq:P-definition}. More precisely, by \eqref{eq:P-first-estimate}, it suffices to estimate the terms in \eqref{eq:I1-I2-I3} and, when $m>0$, in \eqref{eq:Im1-Im2}.
We provide two different bounds according to the dimension, treating first the higher-dimensional case $d\geq 4$ and then $d=3$.
\begin{proposition}\label{proposition:P-bound-higher}
Let $d\geq 4,$ $m>0$ and let $P$ the potential term defined in~\eqref{eq:P-definition}. Assume that hypotheses~\eqref{eq:hyp-higher-m} hold. Then one has 
\begin{multline}\label{eq:final-P-higher}
|P|	\leq \left(
	\frac{3}{4}C_1c_\epsilon+\frac{3(d-1)}{16}C_2c_\epsilon
	+\frac{3}{2}C_2 (c_{\epsilon/2})^2+\frac{3}{8}m C_4 c_\epsilon
	\right)\vertiii{\nabla\psi}^2\\
	+
	\left(
	\frac{3}{4}C_1c_\epsilon+\frac{3(d-1)}{16}C_2c_\epsilon
	+\frac{3}{8}m C_4 c_\epsilon+\frac{3}{8}m C_3 
	\right)\|\psi\|^2_{L^2(|x|^{-3},dx)}.
\end{multline}
%
where $C_1,C_2, C_3, C_4$ are the constants appearing in~\eqref{eq:hyp-higher-m} and $c_\varepsilon$ is defined in~\eqref{eq:ceps}.
\end{proposition}

From the previous result it follows immediately the corresponding result in the massless case.
\begin{proposition}\label{proposition:P-bound-higher-mzero}
Let $d\geq 4,$ $m=0$ and let $P$ the potential term defined in~\eqref{eq:P-definition}. Assume that hypotheses~\eqref{eq:hyp-higher-mzero} hold. Then one has 
 
\begin{multline}\label{eq:final-P-higher-mzero}
|P|	\leq \left(
	\frac{3}{4}C_1c_\epsilon+\frac{3(d-1)}{16}C_2c_\epsilon
	+\frac{3}{2}C_2 (c_{\epsilon/2})^2
	\right)\vertiii{\nabla\psi}^2
	+
	\left(
	\frac{3}{4}C_1c_\epsilon+\frac{3(d-1)}{16}C_2c_\epsilon
	\right)\|\psi\|^2_{L^2(|x|^{-3},dx)}.
\end{multline}

where $C_1$ and $C_2$ are the constants appearing in~\eqref{eq:hyp-higher-mzero} and $c_\varepsilon$ is defined in~\eqref{eq:ceps}.
\end{proposition}

Now we are in position to prove Proposition~\ref{proposition:P-bound-higher}.
\begin{proof}[Proof of Proposition~\ref{proposition:P-bound-higher}]
In view of \eqref{eq:P-first-estimate}, it suffices to estimate the terms in~\eqref{eq:I1-I2-I3} and~\eqref{eq:Im1-Im2}.

	
	Let us start estimating $I_1.$ Using  the assumption~\eqref{eq:hyp-higher-m} and then the Cauchy-Schwarz inequality, one has 
\begin{equation*}
\begin{split}
	I_1&= \int_{\R^d} |\nabla \psi| |\nabla V| |\psi|\, dx
	\leq C_1 \int_{\R^d} \frac{|\nabla\psi|\,|\psi|}{|x|^{2-\epsilon}+|x|^{2+\epsilon}}\, dx
= C_1 \int_{\R^d} \frac{|\nabla \psi|}{|x|^{1/2-\varepsilon} + |x|^{1/2+ \varepsilon}} \frac{|\psi|}{|x|^{3/2}}\, dx\\ 
&\leq C_1 \Big(\int_{\R^d} \frac{|\psi|^2}{|x|^3}\, dx\Big)^{1/2} \Big( \int_{\R^d} \frac{|\nabla \psi|^2}{(|x|^{1/2-\varepsilon} + |x|^{1/2+ \varepsilon})^2}\, dx\Big)^{1/2}\\
&\leq C_1\, c_\varepsilon \vertiii{\nabla \psi} \Big(\int_{\R^d} \frac{|\psi|^2}{|x|^3}\, dx\Big)^{1/2},
\end{split}
\end{equation*}	
where in the last estimate we have used~\eqref{eq:weighted-to-Morrey}.

Let us continue with $I_2$. Using hypothesis~\eqref{eq:hyp-higher-m} one gets 
\begin{equation*}
\begin{split}
I_2&=\int_{\R^d} |[\alpha_j, V]| |\nabla \psi| \tfrac{|\psi|}{|x|}\, dx
\leq C_2 \int_{\R^d} \frac{|\nabla \psi|}{(|x|^{1/2-\varepsilon} + |x|^{1/2+\varepsilon})} \frac{|\psi|}{|x|^{3/2}}\, dx\\
&\leq C_2\, c_\varepsilon \vertiii{\nabla \psi} \Big(\int_{\R^d} \frac{|\psi|^2}{|x|^3}\, dx\Big)^{1/2}.
\end{split}	
\end{equation*}
Now we estimate $I_3.$ Once again by \eqref{eq:hyp-higher-m} and by the Cauchy-Schwartz inequality
\begin{equation*}
\begin{split}
	I_3&=\int_{\R^d} |[\alpha_j,V]| |\nabla \psi|^2\, dx 
	\leq C_2\, \int_{\R^d}\frac{|\nabla \psi|^2}{|x|^{1-\varepsilon} + |x|^{1+\varepsilon}} \, dx\leq 2 C_2\, \int_{\R^d}\frac{|\nabla \psi|^2}{(|x|^{1/2-\varepsilon/2} + |x|^{1/2+\varepsilon/2})^2} \, dx\\
	&\leq 2 C_2\,  (c_{\varepsilon/2})^2  \vertiii{\nabla \psi}^2,
	\end{split}
\end{equation*}
where in the last inequality we have used again~\eqref{eq:weighted-to-Morrey}.

Now we  estimate the terms which depend on the mass $m.$ The term $I_{m,1}$ is essentially the term $I_1$ with $|\nabla V|$ replaced by $|[\beta, V]|$, both satisfying the same decay estimates in~\eqref{eq:hyp-higher-m}. Then, reasoning as above,  one has
\begin{equation*}
I_{m,1} \leq C_4 c_\epsilon\vertiii{\nabla \psi}\left(\int_{\R^d}  \frac{|\psi|^2}{|x|^3}\,dx\right)^{1/2}.
\end{equation*}
Finally, using hypothesis \eqref{eq:hyp-higher-m} one can estimate $I_{m,2}$ as follows:
\begin{equation*}
	I_{m,2}=\int_{\R^d} |\partial_r \{V,\beta\}||\psi|^2\, dx
	\leq C_3 \int_{\R^d} \frac{|\psi|^2}{|x|^3}\, dx.
\end{equation*}
Gathering all the above estimates together with~\eqref{eq:P-first-estimate} and the elementary inequality $|ab|\leq \tfrac{a^2}{2}+\tfrac{b^2}{2}$ gives the thesis.
\end{proof}

We shall give now the corresponding bounds in $d=3,$ as for the higher dimensional case, we shall state first the result in the massive case, as the result for $m=0$ then follows immediately from this.

\begin{proposition}\label{proposition:P-bound-3d}
Let $d=3,$ $m>0$ and let $P$ the potential term defined in~\eqref{eq:P-definition}. Assume that hypotheses~\eqref{eq:hyp-higher-m-3d} hold. Then one has 
\begin{multline}\label{eq:final-P-3d}
|P|\leq 
\left(
	\frac{3\sqrt{2}}{4}C_1(c_{\epsilon/2})^2
	+\frac{3\sqrt{2}+12}{8}C_2(c_{\epsilon/2})^2
	+\frac{3\sqrt{2}\, m}{8} C_4(c_{\epsilon/2})^2
	\right)\vertiii{\nabla\psi}^2
	\\
	+
	\left(
	\frac{3\sqrt{2}}{4}C_1(c_{\epsilon/2})^2
	+\frac{3\sqrt{2}}{8}C_2(c_{\epsilon/2})^2
	+\frac{3\sqrt{2}\, m}{8} C_4(c_{\epsilon/2})^2
	+\frac{3 m}{8} C_3(c_{\epsilon/2})^2
	\right)
	\|\psi\|_R^2,
\end{multline}
where $C_1,C_2, C_3, C_4$ are the constants appearing in~\eqref{eq:hyp-higher-m} and $c_\varepsilon$ is defined in~\eqref{eq:ceps}.
\end{proposition}

From the previous result it follows immediately the corresponding result in the massless case.
\begin{proposition}
Let $d=3,$ $m=0$ and let $P$ the potential term defined in~\eqref{eq:P-definition}. Assume that hypotheses~\eqref{eq:hyp-higher-mzero-3d} hold. Then one has 
\begin{equation}\label{eq:final-P-3d-massless}
|P|\leq 
\left(
	\frac{3\sqrt{2}}{4}C_1(c_{\epsilon/2})^2
	+\frac{3\sqrt{2}+12}{8}C_2(c_{\epsilon/2})^2
	\right)\vertiii{\nabla\psi}^2
	+
	\left(
	\frac{3\sqrt{2}}{4}C_1(c_{\epsilon/2})^2
	+\frac{3\sqrt{2}}{8}C_2(c_{\epsilon/2})^2
	\right)
	\|\psi\|_R^2,
\end{equation}
where $C_1$ and $C_2$ are the constants appearing in~\eqref{eq:hyp-higher-m} and $c_\varepsilon$ is defined in~\eqref{eq:ceps}.
\end{proposition}

We shall prove now Proposition~\ref{proposition:P-bound-3d}.
\begin{proof}[Proof of Proposition~\ref{proposition:P-bound-3d}]
As in the proof of Proposition~\ref{proposition:P-bound-higher},
in view of \eqref{eq:P-first-estimate}, it suffices to estimate the terms in~\eqref{eq:I1-I2-I3} and~\eqref{eq:Im1-Im2}. 

We begin with $I_1$. By~\eqref{eq:hyp-higher-m-3d} and the Cauchy-Schwartz inequality, one has
\begin{equation*}
\begin{split}
	I_1
	&=
	\int_{\R^d} |\nabla \psi| |\nabla V| |\psi|\,dx
	\leq C_1\int_{\R^d}	\frac{|\nabla\psi|\,|\psi|}{|x|^{2-\epsilon}+|x|^{2+\epsilon}}\,dx\\
	& \leq 2C_1\left(\int_{\R^d}	
	\frac{|\nabla\psi|^2}{\left(|x|^{1/2-\epsilon/2}+|x|^{1/2+\epsilon/2}\right)^2}\,dx \right)^{1/2}
	 \left(\int_{\R^d}	\frac{|\psi|^2}{\left(|x|^{3/2-\epsilon/2}+|x|^{3/2+\epsilon/2}\right)^2}\,dx \right)^{1/2}
			\\
	&\leq \sqrt{2}
	C_1(c_{\epsilon/2})^2\vertiii{\nabla\psi}
	\,{\|\psi\|}_R
\end{split}
\end{equation*}
where in the last estimate we have used \eqref{eq:weighted-to-Morrey} and \eqref{eq:weighted-to-spherical}.

The term $I_2$ is essentially the term $I_1$ with $|\nabla V|$ replaced by $|[\alpha_j,V]|/|x|$ both satisfying the same decay estimate in \eqref{eq:hyp-higher-m-3d}. Then, reasoning as above, we conclude
\begin{equation*}
	I_2\leq \sqrt{2}C_2(c_{\epsilon/2})^2\vertiii{\nabla\psi}
	\,{\|\psi\|}_R.
\end{equation*}
Finally, the estimate of $I_3$ follows by repeating the argument of Proposition~\ref{proposition:P-bound-higher}, yielding
\begin{equation*}
	I_3\leq 2 C_2\,  (c_{\varepsilon/2})^2  \vertiii{\nabla \psi}^2.
\end{equation*}
Now we will estimate the terms which depend on the mass $m.$ The term $I_{m,1}$ is essentially the term $I_1$ with $|\nabla V|$ replaced by $|[\beta, V]|$, both satisfying the same decay estimates in~\eqref{eq:hyp-higher-m}. Then, reasoning as above,  one has
\begin{equation*}
	I_{m,1}\leq \sqrt{2}C_4(c_{\epsilon/2})^2\vertiii{\nabla\psi}
	\,{\|\psi\|}_R.
\end{equation*}
Finally, let us focus on $I_{m,2}$. Using~\eqref{eq:hyp-higher-m-3d} and the Cauchy-Schwartz inequality, one has
\begin{equation*}
\begin{split}
I_{m,2}&=\int_{\R^d} |\partial_r \{V,\beta\}| |\psi|^2\,dx
\leq C_3 \int_{\R^d} \frac{|\psi|^2}{|x|^{3-\epsilon}+|x|^{3+\epsilon}}\,dx
\leq 2C_3 \int_{\R^d} \frac{|\psi|^2}{(|x|^{3/2-\epsilon/2}+|x|^{3/2+\epsilon/2})^2}\,dx\\
&\leq C_3 (c_{\epsilon/2})^2 \|u\|_R^2.
\end{split}
\end{equation*}
Gathering all the above estimates together with~\eqref{eq:P-first-estimate} with $d=3$ and the elementary inequality $|ab|\leq \tfrac{a^2}{2}+\tfrac{b^2}{2}$ gives the thesis.
\end{proof}

\subsection{Time-derivative}
This section will provide an estimate for the following quantity
\begin{equation*}
T:=\frac{d}{dt} \Im \langle H_D \psi, iG_\phi \psi \rangle,
\end{equation*}
which appears in the left-hand-side of~\eqref{eq:gen-id}. Here $\phi$ is defined as in~\eqref{eq:multiplier-choice}.

We shall start with the following lemma.
\begin{lemma}
Let $\phi$ be the multiplier defined in~\eqref{eq:multiplier-choice}. Then there exists a positive constant  $C_d>0,$ that depends on the dimension, such that
\begin{equation}\label{eq:im-nabla}
	|\Im \langle H_D \psi, iG_\phi \psi \rangle|\leq C_d (\|\nabla \psi\|_{L^2}^2 + m^2 \|\psi\|_{L^2}^2).
\end{equation}
\begin{proof}
Combining~\eqref{eq:iGphi-computed} with \eqref{eq:estimate-phiprime-phipprime} and the Hardy inequality~\eqref{eq:classical-Hardy}, one obtains 
\begin{equation*}
	\|iG_\phi \psi\|_{L^2}
	\leq \left (\frac{3(d-1)}{4(d-2)} + \frac{3}{4} \right) \|\nabla \psi\|_{L^2}^2=\frac{3(2d-3)}{4(d-2)}\|\nabla \psi\|_{L^2}^2
\end{equation*}
Using the Cauchy-Schwarz inequality and the bound above one has 
\begin{equation*}
\begin{split}
	|\Im \langle H_D \psi, iG_\phi \psi \rangle|
	&\leq \|\nabla \psi\|_{L^2} \|iG_\phi \psi\|_{L^2} + m \| \psi\|_{L^2} \|iG_\phi \psi\|_{L^2}\\
	&\leq c_d \Big[ \|\nabla \psi\|_{L^2}^2  + m \| \psi\|_{L^2} \|\nabla \psi\|_{L^2}\Big]\\
	&\leq \frac{3}{2} c_d \Big[ \|\nabla \psi\|_{L^2}^2  + m^2 \| \psi\|_{L^2}^2\Big].
\end{split}
\end{equation*}
here we have defined 
\begin{equation*}
	c_d=\frac{3(2d-3)}{4(d-2)}.
\end{equation*}
Then the thesis follows with $C_d=\tfrac{3}{2}c_d.$
\end{proof}
\end{lemma}

\subsection{Energy Bounds}
\label{sec:almost-conservation}
Before proving our main results Theorem~\ref{thm:3-dim-massless}--Theorem~\ref{thm:higher-dim-massive}, we observe that, even though, in principle, if $V\neq 0$ the kinetic energy $\|\nabla \psi(t)\|_{L^2}$ is not conserved, suitable assumptions on the potential give some \emph{a priori} estimates on the $H_D$-norm, as made precise in the following result.


\begin{proposition}[A priori estimate of the $H_D$-norm]\label{lemma:energy-bound}
Let Assumption~\ref{ass:HE} be satisfied.
%
Then 
any solution $\psi$ of~\eqref{eq:dirac-evolution} satisfies
\begin{equation}\label{eq:energy-bound}
	\|H_D \psi(t,\cdot)\|_{L^2}^2=
	\|\nabla \psi(t,\cdot)\|_{L^2}^2+m^2\| \psi(t,\cdot)\|_{L^2}^2\leq \tfrac{1}{(1-a)^2} \|(H_D+V) \psi(0,\cdot)\|_{L^2}^2
\end{equation} 
\end{proposition}
Before proving this result we shall need some preliminary results.

\begin{proposition}\label{prop:conservation-laws}
Let Assumption~\ref{ass:HE} be satisfied and let $\psi\in \mathcal{D}(H)$ be a solution to~\eqref{eq:dirac-evolution}.
Then one has 
	\begin{equation}\label{eq:mass-conservation}
		M(t):= \int_{\R^d} |\psi(t,x)|^2\, dx= M(0)
		\qquad \qquad \text{(Conservation of mass)}
	\end{equation}
	\begin{equation}\label{eq:cons-Hamiltonian}
		H(t):= \int_{\R^d} |(H_D+ V)\psi(t,x)|^2\, dx= H(0)
		\qquad \qquad \text{(Conservation of Hamiltonian)}
	\end{equation}
\end{proposition}

\begin{proof} 
We begin by observing that, under Assumption~\ref{ass:HE}, for any $\psi\in \D(H_D+V)$ one has that both $H_D\psi$ and $V\psi$ belong to $L^2(\R^d;\C^N).$ This makes it rigorous the equation manipulations performed below.

We now turn to the proof of the mass conservation stated in~\eqref{eq:mass-conservation}.
Choosing $\psi$ as a test function in the weak formulation of~\eqref{eq:dirac-evolution} and considering the imaginary part of the resulting identity, we obtain
\begin{equation}\label{eq:id-cons-mass}
	\Im\langle i\partial_t \psi, \psi \rangle= \Im \langle H_D \psi, \psi \rangle + \Im \langle V\psi, \psi \rangle. 
\end{equation}	
Since both $H_D$ and $V$ are self-adjoint, both the scalar products on the right-hand side of~\eqref{eq:id-cons-mass} are real, therefore the right-hand side is zero. On the other hand, one can easily see that 
\begin{equation*}
\Im \langle i\partial_t \psi, \psi\rangle = \Re \langle \partial_t \psi, \psi\rangle = \frac{1}{2} \frac{d}{dt} \langle \psi, \psi \rangle,
\end{equation*}
this then gives the conservation of mass.

Let us now focus on the Hamiltonian conservation~\eqref{eq:id-other-cons}. First we approximate our initial datum $f\in \mathcal{D}(H)$ with $f_j\in C^{\infty}_c,$ then the corresponding solution $\psi_j=e^{itH}f_j\in C(\R; H^{3/2}(\R^d)).$ For the sake of  readability, we drop the subscript $j$ in the following.
Under our assumptions, one can show that $\partial_t(H_D+V)\psi\in H^{-1/2}(\R^d).$ Multiplying~\eqref{eq:dirac-evolution} by this test function and taking the real part of the resulting identity, one obtains
\begin{equation}
\label{eq:id-other-cons} 
	\Re \langle i\partial_t \psi, \partial_t(H_D+V) \psi \rangle
	= \Re \langle (H_D +V)\psi, \partial_t(H_D+V) \psi \rangle =\frac{1}{2} \frac{d}{dt} \|(H_D+V)\psi\|^2.
\end{equation}
Here the brackets are understood as duality pairing between $H^{1/2}$ and $H^{-1/2}.$
Since $(H_D+V)$ is self-adjoint and commutes with $\partial_t$, the scalar product $\langle \partial_t \psi, \partial_t(H_D+V) \psi \rangle$ is real, hence
$\Re \langle i\partial_t \psi, (H_D+V) \partial_t \psi \rangle= \Re i \langle \partial_t \psi, (H_D+ V)\partial_t \psi \rangle=-\Im\langle \partial_t \psi, (H_D+V) \partial_t \psi \rangle=0.$ 
Substituting into~\eqref{eq:id-other-cons}, one obtains
\begin{equation}\label{eq:cons-comm}
    \frac{1}{2}\frac{d}{dt} \| (H_D+ V)\psi\|=0.
\end{equation}
This gives~\eqref{eq:cons-Hamiltonian} for the approximate solution $\psi_j.$ The identity for the original solution is then obtained passing to the limit.
\end{proof}
\begin{remark}
The Hamiltonian conservation~\eqref{eq:cons-Hamiltonian} can be obtained as a straightforward consequence of the following more general result.
Let $K$ be a self-adjoint, time-independent operator such that $K\partial_t\psi\in \D(H_D+V)$ and
$[K, H_D + V]=0$.
Then $\langle (H_D + V)\psi(t,\cdot), K\psi(t,\cdot) \rangle$ is conserved. The proof follows the same argument: one chooses $K\partial_t\psi$ as a test function in~\eqref{eq:dirac-evolution} and takes the real part of the resulting identity.
\end{remark}

\medskip
We are now ready to prove Proposition~\ref{lemma:energy-bound}

\begin{proof}[Proof of Proposition~\ref{lemma:energy-bound}]
By~\eqref{eq:commutation-rel-dirac-matr} and ~\eqref{eq:smallness-potentia} we have 
\[
\|V\psi(t,\cdot)\|_{L^2} \leq a\|\nabla\psi(t,\cdot)\|_{L^2}
\leq 
a\sqrt{\|\nabla\psi(t,\cdot)\|_{L^2}^2+m^2\|\psi(t,\cdot)\|_{L^2}^2}
=a\|H_D\psi(t,\cdot)\|_{L^2}
\]
Thanks to this and by the triangular inequality, we obtain
\[
\|H_D\psi(t,\cdot)\|_{L^2}
\leq
\|(H_D+V)\psi(t,\cdot)\|_{L^2}+\|V\psi(t,\cdot)\|_{L^2}
\leq \|(H_D+V)\psi(t,\cdot)\|_{L^2}+a\|H_D\psi(t,\cdot)\|_{L^2}
\]
Combining this with~\eqref{eq:cons-Hamiltonian} concludes the proof.
\end{proof}

\begin{remark}\label{rem:almost-conservation}
When Assumption~\ref{ass:HE} is satisfied, thanks to \eqref{eq:smallness-potentia} we can modify~\eqref{eq:energy-bound} as follows
\[
\|\nabla \psi(t,\cdot)\|_{L^2}^2+m^2\| \psi(t,\cdot)\|_{L^2}^2\leq \tfrac{(1+a)^2}{(1-a)^2}\left( \|\nabla \psi(0,\cdot)\|_{L^2}^2+m^2\| \psi(0,\cdot)\|_{L^2}\right)
\]
This result can be interpreted as an almost conservation of the $H_D$-norm of solutions to~\eqref{eq:dirac-evolution}. 
As a consequence, the term $\|(H_D+V)\psi(0,\cdot)\|_{L^2}^2$ appearing in Theorems~\ref{thm:3-dim-massless}--\ref{thm:higher-dim-massive} can be replaced by $\|\nabla\psi(0,\cdot)\|_{L^2}^2+m^2\|\psi(0,\cdot)\|_{L^2}^2.$
\end{remark}

\subsection{Proof of local-smoothing estimates}

With all the previous tools at hands, we are finally able to prove our results
Theorem~\ref{thm:3-dim-massless} - Theorem~\ref{thm:higher-dim-massive}. We shall start with the result in higher dimensions $d\geq 4,$ namely with Theorem~\ref{thm:higher-dim-massive}.
 
 \begin{proof}[Proof of Theorem~\ref{thm:higher-dim-massive}]
 From identity~\eqref{eq:gen-id} one has
 \begin{equation*}
 	-\frac{d}{dt} \Im \langle H_D \psi, iG_\phi \psi \rangle
 	= K + P, 
 \end{equation*}
 where $\phi$ is defined in~\eqref{eq:multiplier-choice}, $K$ is as in~\eqref{eq:kinetic} and $P$ is defined in~\eqref{eq:P-definition}.
From the bound for $K$ contained in Proposition~\ref{prop:K-lower-bound} one has 
 \begin{equation*}
 \begin{split}
 -\frac{d}{dt} \Im \langle H_D \psi, iG_\phi \psi \rangle 
 &\geq \frac{(d-1)}{4d} \frac{1}{R}  \int_{|x|\leq R} |\nabla \psi|^2\, dx
 + \frac{(d-1)(d-3)}{8} \int_{\R^d} \frac{|\psi|^2}{|x|^3}\, dx \\
	&\phantom{\geq} +\frac{1}{2} \int_{\R^d} \frac{|\partial_\tau \psi|^2}{|x|}\, dx
	+ \frac{(d-1)}{16}\frac{1}{R^2} \int_{|x|=R} |\psi|^2\, d\sigma
	- |P|.
\end{split}
 \end{equation*}
 Using the bound~\eqref{eq:final-P-higher} for $|P|$ contained in Proposition~\ref{proposition:P-bound-higher} one gets
\begin{equation*}
\begin{split}
 	 -\frac{d}{dt} \Im \langle H_D \psi, iG_\phi \psi \rangle 
 &\leq \delta \frac{1}{R} \int_{|x|\leq R} |\nabla \psi|^2\, dx
 + \delta \int_{\R^d} \frac{|\psi|^2}{|x|^3}\, dx\\
	&\phantom{\geq }+\frac{1}{2} \int_{\R^d} \frac{|\partial_\tau \psi|^2}{|x|}\, dx
	+ \frac{(d-1)}{16}\frac{1}{R^2} \int_{|x|=R} |\psi|^2\, d\sigma,
\end{split}
\end{equation*} 
where $\delta$ is an explicit strictly positive constant, whose value can be traced from the proof.

Integrating in time the previous estimate, taking the supremum over $R>0,$ using~\eqref{eq:im-nabla} together with the energy bound~\eqref{eq:energy-bound} and the mass conservation~\eqref{eq:mass-conservation}, one has
\begin{multline*}
 	 \delta  \sup_{R>0} \frac{1}{R} \int_0^\infty \int_{|x|\leq R} |\nabla \psi|^2\, dx\, dt
 	 + \delta \int_0^\infty \int_{\R^d} \frac{|\psi|^2}{|x|^3}\, dx\, dt\\
 	 + \frac{1}{2}\int_0^\infty\int_{\R^d} \frac{|\partial_\tau \psi|^2}{|x|}\, dx\, dt 	
 	 + \frac{(d-1)}{16} \sup_{R>0} \frac{1}{R^2} \int_0^\infty\int_{|x|=R} |\psi|^2\, d\sigma \, dt
 	 \lesssim_d \| H\psi(0,\cdot) \|_{L^2}^2, 
\end{multline*}  
which is the thesis.
 \end{proof}
 
 \begin{proof}[Proof of Theorem~\ref{thm:higher-dim-massless}, Theorem~\ref{thm:3-dim-massive} and Theorem~\ref{thm:3-dim-massless}]
 The proof of those results follows the same steps as the previous proof, simply one has to use a different bound for $|P|$ than~\eqref{eq:final-P-higher}, namely~\eqref{eq:final-P-higher-mzero}, ~\eqref{eq:final-P-3d} and~\eqref{eq:final-P-3d-massless}, respectively.   
 \end{proof}

\appendix
\section{Integral Estimates}
This appendix is devoted to the proof of~\eqref{eq:weighted-to-Morrey} and~\eqref{eq:weighted-to-spherical}.
\label{appendix:1}
\begin{proof}[Proof of~\eqref{eq:weighted-to-Morrey} and~\eqref{eq:weighted-to-spherical}]

We recall that for any integrable function $f$
\begin{equation}\label{eq:annulus-int}
    \int_{\R^d} f\, dx = \sum_{j\in \mathbb{Z}} \int_{2^j\leq |x|\leq 2^{j+1}} f\, dx.
\end{equation}
Thanks to this and since the function $r\in (0,+\infty)\mapsto (r^{1/2-\delta}+r^{1/2+\delta})^{-2}$ is decreasing for $0<\delta<1/2$, we can rewrite the right-hand side of~\eqref{eq:weighted-to-Morrey} as
	\begin{equation*}
	\begin{split}
		\int_{\R^d} \frac{|u|^2}{(|x|^{1/2-\delta} + |x|^{1/2 + \delta})^2}dx
		&=
		\sum_{j\in \Z} \int_{2^{j}\leq |x|\leq 2^{j+1}}\frac{|u|^2}{(|x|^{1/2-\delta} + |x|^{1/2 + \delta})^2}dx\\
		&\leq 
		\sum_{j\in\Z} \frac{2^{j+1}}{((2^j)^{1/2-\delta} + (2^j)^{1/2 + \delta})^2}\cdot 
		 \frac{1}{2^{j+1}}\int_{|x|\leq 2^{j+1}}|u|^2 dx\\
&\leq c_\delta^2 \vertiii{u}^2,
		\end{split}
	\end{equation*}
	where we have defined 
	\begin{equation}\label{eq:ceps}
		c_\delta^2:= \sum_{j\in\Z} \frac{2}{(2^{-\delta j}+2^{\delta j})^2}\approx \frac{2}{\delta\log(2)}.
	\end{equation}
	Let us now consider \eqref{eq:weighted-to-spherical}. Reasoning as above and using the coarea formula we have
	\begin{equation*}
	\begin{split}
\int_{\R^d} \frac{|u|^2}{(|x|^{3/2-\delta} + |x|^{3/2+\delta})^2}\,dx
	    &=
	    \sum_{j\in\mathbb{Z}}
	    \int_{2^j}^{2^{j+1}}\frac{1}{\left(\rho^{1/2-\delta}+\rho^{1/2+\delta}\right)^2}
	    \cdot\left(\frac{1}{\rho^2}\int_{|x|=\rho} |u|^2\,d\sigma\right)d\rho
	    \\
	    &\leq
	    \|u\|_R^2
	    \sum_{j\in\Z} \frac{2^{j+1}-2^j}{((2^j)^{1/2-\delta} + (2^j)^{1/2 + \delta})^2}= \frac{c_\delta^2}{2} \|u\|_R^2.
	\end{split}
	\end{equation*}
	This concludes the proof of~\eqref{eq:weighted-to-Morrey}.
\end{proof}

\section*{Acknowledgements}

L.C. is supported by the grant Ramón y Cajal RYC2021-032803-I  funded by MCIN/AEI/10.13039/50110 and by Ikerbasque.

L.C, L.F. and F.P. are partially supported by the project PID2024-155550NB-I00 funded by MICIU/AEI/10.130	\\
39/501100011033 and by ERDF/EU. 

L.F. and F.P. are partially supported by the Basque Government through the BERC 2022-2025 program, by the Ministry of Science and Innovation: BCAM Severo Ochoa accreditation CEX2021-001142-S/MICIN/AEI/
10.13039/ 501100011033 PID2021-123034NB-I00 funded by MCIN/AEI/10.13039/501100011033

/FEDER,UE. 

L.F. is also supported by the projects IT1615-22 funded by the Basque Government, and by Ikerbasque.



\bigskip

\bibliographystyle{amsalpha}

\end{document}